\documentclass[11pt]{amsart}

\title[Asymptotics of scalar waves]{Asymptotics of scalar waves on
  long-range asymptotically Minkowski spaces}
\author{Dean Baskin}
\address{Department of Mathematics, Texas A\&M University}
\email{dbaskin@math.tamu.edu}
\author{Andr{\'a}s Vasy}
\address{Department of Mathematics, Stanford University}
\email{andras@math.stanford.edu}
\author{Jared Wunsch}
\address{Department of Mathematics, Northwestern University}
\email{jwunsch@math.northwestern.edu}

\usepackage{amsmath, amssymb, verbatim, mathrsfs}
\usepackage{graphicx}

\newtheorem{theorem}{Theorem}[section]
\newtheorem{lemma}[theorem]{Lemma}
\newtheorem{proposition}[theorem]{Proposition}

\newtheorem*{maintheorem}{Theorem \ref{mainthm}}
\numberwithin{equation}{section}

\newtheorem{non-theorem}[theorem]{Non-Theorem}

\theoremstyle{remark}
\newtheorem{definition}[theorem]{Definition}
\newtheorem{remark}[theorem]{Remark}

\newcommand\scl{{\mathrm{sc}}}
\newcommand\bl{{\mathrm{b}}}

\newcommand{\hamvf}{\mathsf{H}}

\newcommand{\B}{\mathfrak{B}}

\newcommand{\E}{\mathcal{E}}
\newcommand{\Etot}{\mathcal{E}_{\text{tot}}}

\newcommand{\cM}{\mathcal{M}}
\newcommand{\cI}{\mathcal{I}}
\newcommand{\sH}{\hamvf}
\newcommand{\cR}{\mathcal{R}}
\newcommand{\cX}{\mathcal{X}}
\newcommand{\cY}{\mathcal{Y}}

\newcommand{\CI}{\mathcal{C}^\infty}
\newcommand{\dCI}{\dot{\mathcal{C}}^\infty}

\newcommand{\CmI}{\mathcal{C}^{-\infty}}

\newcommand{\tow}{\mathrm{ftr}}
\newcommand{\away}{\mathrm{past}}

\newcommand{\Mlog}{\mathrm{M}}
\newcommand{\badmod}{\mathcal{N}}

\newcommand{\bdflog}{\varrho}
\newcommand{\rholog}{\bdflog}
\newcommand{\vlog}{\mathrm{v}}
\newcommand{\ylog}{\mathrm{y}}
\newcommand{\CIlog}{\CI_{\log}}
\newcommand{\cMlog}{\mathcal{M}_{\log}}
\newcommand{\cIlog}{\mathcal{I}_{\log}}
\newcommand{\Diff}{\operatorname{Diff}}

\newcommand{\Diffblogbad}{\widetilde{\operatorname{Diff}}_{\bl,\log}}
\newcommand{\Diffb}{\operatorname{Diff}_{\bl}}
\newcommand{\Diffblog}{\operatorname{Diff}_{\bl, \log}}
\newcommand{\cMdiff}{\cM_{\mathrm{D}}}
\newcommand{\cMdifflog}{\cM_{\mathrm{D}, \log}}
\newcommand{\Olog}{O_{\log}}
\newcommand{\Rlog}{\mathrm{R}}
\newcommand{\taut}{\iota}
\newcommand{\eu}{\overline{\cup}}
\newcommand{\phg}[1]{\mathcal{A}^{#1}_{\text{phg}}}

\newcommand{\rad}{\mathbf{R}}
\newcommand{\err}{\mathcal{O}}
\newcommand{\vv}{v_0}
\newcommand{\ggamma}{\gamma_0}
\newcommand{\omegabar}{\overline{\omega}}
\newcommand{\alphabar}{\overline{\alpha}}
\newcommand{\betabar}{\overline{\beta}}

\newcommand{\scri}{\mathscr{I}}

\newcommand{\mass}{m}

\newcommand{\resset}{\mathcal{E}_{\mathrm{res}}}
\newcommand{\masslessres}{\mathcal{E}_{\mathrm{res}}^0}
\newcommand{\ressetinit}{\mathcal{E}_0}

\newcommand{\smoothset}{\mathcal{E}_{\mathrm{smooth}}}
\newcommand{\logset}{\mathcal{E}_{\log}}
\newcommand{\logsmoothset}{\mathcal{E}_{\CIlog}}

\newcommand{\Tscstar}{{}^{\scl}T^*}

\newcommand{\Tbstar}{{}^{b}T^*}
\newcommand\Tsc{{}^{\scl} T}
\newcommand\Tb{{}^{\bl} T}
\newcommand\Sb{{}^{\bl} S}
\newcommand\Vf{\mathcal{V}}
\newcommand\Vb{\mathcal{V}_\bl}
\newcommand\Vsc{\mathcal{V}_\scl}
\newcommand{\sigmab}{\sigma_\bl}
\newcommand{\sigmasc}{\sigma_{\scl}}
\newcommand{\Diffsc}{\Diff_{\scl}}
\newcommand{\Psib}{\Psi_\bl}
\newcommand{\Hb}{H_\bl}
\newcommand{\WFb}{\operatorname{WF}_\bl}
\newcommand{\xisc}{\xi_{\scl}}
\newcommand{\gammasc}{\gamma_{\scl}}
\newcommand{\etasc}{\eta_{\scl}}

\newcommand{\abs}[1]{{\left\lvert{#1}\right\rvert}}
\newcommand{\smallabs}[1]{{\lvert{#1}\rvert}}
\newcommand{\norm}[1]{{\left\lVert{#1}\right\rVert}}
\newcommand{\Norm}[2][]{\left\|{#2}\right\|_{#1}}

\newcommand{\smallang}[1]{{\langle{#1}\rangle}}

\renewcommand{\Im}{\operatorname{Im}}
\renewcommand{\Re}{\operatorname{Re}}
\newcommand{\im}{\operatorname{Im}}

\newcommand{\supp}{\operatorname{supp}}
\newcommand{\Op}{\operatorname{Op}}

\newcommand{\pa}{{\partial}}
\newcommand{\pd}[1][]{\partial_{#1}}
\newcommand{\ep}{{\epsilon}}

\DeclareMathOperator{\Ann}{Ann}
\newcommand{\RR}{\mathbb{R}}
\newcommand{\reals}{\RR}

\newcommand{\NN}{\mathbb{N}}

\newcommand{\CC}{\mathbb{C}}

\newcommand{\Mellin}{\mathcal{M}} 
\newcommand{\hol}{\mathcal{H}}

\thanks{The authors acknowledge partial support from NSF grants
  DMS-1500646 (DB), DMS-1361432 (AV) and DMS-1001463 (JW), and the
  support of NSF Postdoctoral Fellowship DMS-1103436 (DB).  The
  authors gratefully acknowledge the hospitality of the Erwin
  Schr\"odinger Institute program ``Modern Theory of Wave Equations,''
  at which some of this work was carried out in summer 2015.  The
  first and third author also thank the Institut Henri Poincar\'e for
  support through its ``Research in Paris'' program in February 2016.}

\date{January 8, 2018}


\begin{document}

\begin{abstract}
  We show the existence of the full compound asymptotics of solutions
  to the scalar wave equation on long-range non-trapping Lorentzian
  manifolds modeled on the radial compactification of Minkowski space.
  In particular, we show that there is a joint asymptotic expansion at
  null and timelike infinity for forward solutions of the
  inhomogeneous equation.  In two appendices we show how these results
  apply to certain spacetimes whose null infinity is modeled on that
  of the Kerr family.  In these cases the leading order logarithmic term
  in our asymptotic expansions at null infinity is shown to be nonzero.
\end{abstract}

\maketitle

\section{Introduction}
\label{sec:introduction}

In this paper we analyze the full compound asymptotics of solutions to
the scalar wave equation on long-range non-trapping Lorentzian
scattering manifolds.  This class of Lorentzian scattering manifolds,
introduced in~\cite{radfielddecay}, includes short-range perturbations
of the Minkowski spacetimes as well as a broad class of rather
different spacetimes that admit a compactification analogous to the
spherical compactification of Minkowski space.  In this paper we
extend these results to the more physically meaningful setting of
\emph{long-range} perturbations of gravitational type: this entails
adding a term to our metric that involves a constant Bondi mass.
We analyze the compound asymptotics of scalar waves
near the boundary at infinity.  The most interesting region for this
expansion is near the boundary of the light cone, where we obtain a
full understanding of the asymptotics via an appropriately scaled
blow-up; the \emph{front face} of this blow-up, i.e., the new boundary
face obtained by introduction of polar coordinates, is $\scri^+$, the
null infinity of our spacetime.  We analyze the Friedlander radiation
field, which is given by the restriction of the rescaled solution to
$\scri^+$; in particular we find as in~\cite{radfielddecay} that the
asymptotics of the radiation field in the ``time-delay'' parameter
(given by $s=2(t-r)$ in Minkowski space and subtler here owing to
long-range effects) are determined by the resonance poles of an
associated Laplace-like operator for an asymptotically hyperbolic
metric on the ``cap'' in the sphere at infinity reached by forward
limits of time-like geodesics.  Among the main differences of the
construction here and that used for the short-range case in
\cite{radfielddecay} is the necessity of a change of $\CI$ structure
on the compactified spacetime, prior to the radiation field blow-up,
in order to construct the correct $\scri^+.$

In particular, in the following theorem, the variable $s$ is analogous
to the ``lapse function'' $2(t-r)$ in Euclidean space; in the long-range
case it is given 
instead by
\begin{equation}
  \label{eq:s-coord}
  s=2(t-r)+m\log r^{-1};
\end{equation}
here the logarithmic correction
has a coefficient, denoted $m,$ related to the long-range asymptotics
permitted in our spacetimes.  The geometric hypotheses of the theorem
are spelled out in detail in Section~\ref{sec:long-range-scatt} below,
and indeed we will restate the theorem in a more precise fashion
in Section~\ref{sec:radiation-field-blow}.
\begin{theorem}\label{mainthm}
Let $(M,g)$ be a non-trapping Lorentzian scattering manifold, and let 
$$
\Box_g u=f
$$
with $u \in \CmI(M),$ $f \in \dCI(M).$  Assume that $u$ is a forward solution.
Then $u$ has a joint polyhomogeneous asymptotic expansion in
$s\to\infty,\ r \to \infty$ (where $r$ and $s$ are as in equation~\eqref{eq:s-coord}
\begin{equation}\label{asymptotic.1}
u \sim r^{-(n-2)/2}\sum_j \sum_{\kappa \leq m_j} \sum_{\ell=0}^\infty
\sum_{\alpha\leq 2 \ell} a_{j\kappa\ell\alpha} s^{-\imath \sigma_j}
(\log s)^{\kappa} (s/r)^\ell (\log(s/r))^\alpha.
\end{equation}
If $m=0$ then only $\alpha=0$ terms appear.
\end{theorem}
The slightly eccentric-looking presentation of the terms in the sum is
motivated by the fact that the variables $s^{-1}$ and $(s/r)$ should
be viewed as defining functions of the two faces of a manifold with
corners obtained by the blowup of the light cone at infinity (depicted
on the right side of Figure~\ref{fig:blowup}).  The exponents
$\sigma_j$ have an explicit description as resonance poles of a family
of operators closely related to the spectral family of the Laplacian
on asymptotically hyperbolic space. The radiation field, which is
conventionally defined \cite{Friedlander:1980} by taking the
$s$-derivative of the restriction to $r=\infty$ of $r^{(n-2)/2} u,$ is
thus well defined and enjoys an asymptotic expansion as $s \to \infty$
with terms (given by taking $\ell=\alpha=0$ in \eqref{asymptotic.1})
of the form $s^{-\imath \sigma_j-1} (\log s)^{\kappa}.$ Note that
because $u$ is a forward solution, $u$ is Schwartz for $s\ll 0$ when $\rho$ is
small, hence the regime $s \to +\infty$ is the only interesting
one. 
\begin{figure}[htp]
  \centering
  \includegraphics{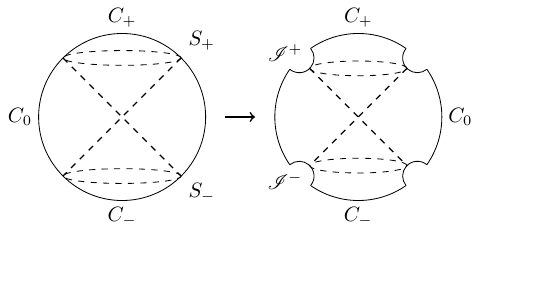}
  \caption{A schematic view of the blow-up.  The lapse function $s$
    increases along $\scri^+$ towards $C_+$.  In the typical Penrose
    diagram of Minkowski space, $C_\pm$ are collapsed to
    $i_\pm$ and $C_0$ is collapsed to $i_0$.}
  \label{fig:blowup}
\end{figure}

This theorem represents an improvement over the results of
\cite{radfielddecay} even in the case $m=0$ (which was all that was
treated in that paper), as the appearance of log terms in the expansion is
considerably clarified.

In practice, the variables in \eqref{asymptotic.1} are not well suited
to the problem, as, for instance, the regime $r, s\to +\infty$ is the
more interesting part of the parameter range.  As in
\cite{radfielddecay} we will in fact view our spacetime as the
interior of a \emph{compact} manifold with boundary $M,$ analogous to
the radial compactification of Minkowski space to the ball $B^{n}.$
We will let $S_\pm$ denote the forward/backward light cones where they
intersect the boundary at infinity, and let $C_\pm$ denote the
interiors of $\pa M$ interior to these light cones, i.e.,  the
future/past timelike infinity, which here is a smooth manifold with
boundary $S_\pm.$
Then $\rho$ denotes a boundary defining function (analogous to
$(1+t^2+r^2)^{-1/2}$) while $v$ will denote a function such that $v=0$
defines $S_+,$ the future light cone at infinity.  Near $S_+$, we can in
fact simply change the $\rho$ variable to $\rho=r^{-1}$ for
simplicity.  Then \emph{in the short-range case}, $s=v/\rho$ and $r =
\rho^{-1}$ are the variables used above, while
$s^{-1}=\rho/v$ and $s/r= v$ will be defining functions for the
boundary faces of the blow-up of $\rho=v=0,$ and the compound
asymptotics \eqref{asymptotic.1} are thus expressed in these
variables.  Note that $\scri^+=\{v=0\}$ while $\rho/v$ tends to $0$ as
we move along $\scri^+$ to forward timelike infinity.  In the
long-range case treated here, these definitions are seriously affected by the
mass parameter, and the necessary changes are addressed
extensively in Section~\ref{sec:logification} below.

\subsection{Notation}
\label{sec:notation}

We use the notation $O(f)$ to denote an element of $f\cdot \CI$ and
the notation $O(f_{1}, \dots, f_{k})$ to denote an element of
$f_{1}\cdot \CI + \dots + f_{k}\cdot \CI$.  We use $\Olog(f)$ and
$\Olog(f_{1},\dots, f_{k})$ similarly (but with $\CIlog$ in place of $\CI$).

Our convention is that the natural numbers include $0:$
$$
\NN \equiv \{0, 1, 2, 3, \dots\}.
$$

\subsection{Sketch of proof}

As the method of proof is in certain respects somewhat round-about, we
sketch it here.  The strategy mimics that developed by the authors for
the short-range case in \cite{radfielddecay} up to a certain point,
where we prove conormality of the solution up to the boundary and to
$S_+.$  The subsequent treatment of the full asymptotic expansion is completely new, however, and much
improves the earlier treatment even in the short-range case.

The main steps in the proof are as follows:
\subsubsection*{Set-up} 

As noted above, we will let $\rho$ be a boundary defining function
for the boundary (at infinity!) of $M$ (e.g., $\rho=r^{-1}$ in the
region of most interest) while $v$ cuts out $S_+,$ the future null
cone (e.g.\ $t/r-1$ in Minkowski space); let $y$ denote the remaining
variables (analogous to $\theta=x/\abs{x} \in S^{n-2}$ when $(t,x) \in
\RR^{1,n-1}$).

We now consider the equation
$$
\Box_g u=f
$$
but then rescale and conjugate\footnote{The conjugation here rescales
  $\Box_{g}$ to a b-operator and makes it formally self-adjoint in the
b-sense.  It also renormalizes the leading asymptotics of solutions of the
wave equation to size $1$.  It is advantageous to work in the b-setting because
the operator is degenerate in the sense of the scattering calculus,
corresponding to
its characteristic set over the boundary being singular at the
zero section, the tip of the light cone in the fibers.} to rewrite it as
$$
L w=g
$$
where
\begin{equation}\label{Ldef}
L\equiv \rho^{-(n-2)/2-2}\Box_g \rho^{(n-2)/2},
\end{equation}
$$
w=\rho^{-(n-2)/2} u\in \CmI(M) ,\quad g=\rho^{-(n-2)/2-2} f \in \dCI(M).
$$
This is advantageous because $L$ is then a ``b-differential operator''
in the sense of Melrose \cite{Melrose:APS}, and enables us to employ
the b-pseudodifferential calculus to obtain microlocal estimates on
$w$ near $X=\pa M.$

\subsubsection*{Propagation of b-regularity}

We first prove a propagation of b-regularity, which is to say
(microlocalized) conormality with respect to the boundary, starting at
the backward null cone, where by hypothesis the solution is trivial (zero near the
boundary), along $X=\pa M.$  This works easily until we reach $S_+,$
where the relevant bicharacteristic flow has radial points, and we
need to use subtler estimates.  The idea is that instead of proving
conormality with respect to $X,$ which is to say, iterated regularity
under $\rho \pa_\rho,$ $\pa_v,$ and $\pa_y,$ we must settle for less:
we only obtain regularity under vector fields that are additionally
tangent to $S_+\subset X$ as well as to $X.$  This is the content of
Proposition~\ref{prop:radial-b-estimate} below.

We also employ a refined version of the b-estimates described above
that have a semiclassical parameter corresponding to the Mellin dual
of the vector field $\rho \pa_\rho.$

\subsubsection*{Fredholm estimates}

The propagation estimates are the necessary ingredient, following the
strategy developed by the second author in
\cite{Vasy-Dyatlov:Microlocal-Kerr}, in showing that we may set up a
global Fredholm problem on $X$ for the family of ``reduced normal operators'' $P_\sigma.$
This is the family of operators given by an appropriate freezing of coefficients at $\rho=0$ and
conjugating $L$ by the Mellin transform in the boundary defining
function (which simply acts on $L$ by replacing the vector field $\rho
D_\rho$ by the parameter $\sigma$ wherever it appears).  Crucially,
$P_\sigma$ is taken to act on spaces with varying degrees of
regularity, with more regularity mandated at the backward (in the
sense of the time orientation) end of the
flow lines than at the forward end.   The
non-trapping hypothesis is used here in a crucial way to show that
the family $P_{\sigma}^{-1}$ (which we show always exists as a
meromorphic operator family) may moreover only have finitely many poles
in a given horizontal strip in $\CC,$ and satisfies polynomial
estimates as $\smallabs{\Re \sigma} \to \infty.$

\subsubsection*{Global asymptotic expansion}

We then begin the asymptotic development of $w$ (and hence $u$) near
the boundary as follows.  Cutting off near $\pa M$ and Mellin
transforming, the equation
$$
Lw=g
$$
becomes a family of equations of the form
$$
P_\sigma \tilde w=\tilde g.
$$
A priori, all we know is that $\tilde w$ is analytic in a half-space
$\Im \sigma \geq \varsigma_0.$ However we may invert $P_\sigma$ to
obtain \emph{meromorphy} of $\tilde w,$ with poles arising from the
poles of $P_{\sigma}^{-1}.$ If $L$ were in fact
dilation-invariant near $\pa M,$ we would immediately have global
meromorphy, but as error terms need to be dealt with at each stage, we
are only able to improve our domain of meromorphy a little at a time,
increasing the half-space in which we know meromorphy by finite
increments in an iterative argument. This iteration does eventually
yield global meromorphy, but with the subtlety that poles may arise
not merely from the poles of $P_\sigma^{-1}$ itself but also from
their shifts by $\imath j$ for $j \in \NN.$

Applying the inverse Mellin transform has the effect of turning poles of $\tilde
w$ into terms in an asymptotic expansion, with a pole at $\sigma=z$ of
degree $k$ becoming a term $\rho^{\imath z} \log \rho^{k-1}.$  The
coefficients of this expansion, however, are functions on $X=\pa M$
that \emph{become worse} in their regularity at $S_+$ as $\Im z$
decreases, i.e.\ as we obtain more decay in $\rho.$    This
development thus suffices
to get an asymptotic expansion valid as $\rho\downarrow 0$ for $v \neq
0,$ and indeed to get this expansion uniformly as $\rho/v \downarrow 0$
near $v=0,$ which is in effect one of the two asymptotic expansions at
intersecting boundary faces in \eqref{asymptotic.1} (where $s^{-1}=\rho/v$
and $v$ should be regarded as the defining functions for two
intersecting boundary faces of a manifold with corners, where we seek
a joint asymptotic expansion).

\subsubsection*{Full asymptotics}

It thus remains to obtain the full expansion \eqref{asymptotic.1}, in
both variables.  It will suffice, via an argument discussed in
Section~\ref{sec:basics-b-geometry}, to show that in the short-range case $w$ has improved
asymptotics under applications of vector fields of the form
$$
(R+\imath j)\dots (R+\imath ) (R)
$$
where $R=vD_v+\rho D_\rho$ is the scaling vector field about $S_+.$
In the long-range case, we must employ instead
$$
(R+\imath j)^{2j+1}\dots (R+\imath )^3 (R),
$$
with the increased multiplicity corresponding to the additional log
terms appearing in the long-range case.  These extra difficulties
arise because of a change of coordinates necessary to obtain good
estimates at $S_+$: if we (locally) replace the variable $v$ that defines $S_+$
with $v+m \rho \log \rho,$ this change of variables has the
considerable virtue of making the leading order form of $L$ near $S_+$
the same in the long- and short-range cases, but also the considerable
defect of introducing additional log singularities into the
coefficients of the remaining terms in $L.$ It is these additional
error terms that are responsible for the additional log singularities
in our expansion.  This change of variables and its consequences (in
particular, a change of $\CI$ structure on the manifold $M$) are
discussed in Section~\ref{sec:logification} below.

\section{Basics of b- and scattering-geometry}
\label{sec:basics-b-geometry}

\subsection{b-geometry}\label{subsec:b-geometry}
The main microlocal tool that we employ is the
\emph{b-pseudodifferential calculus} of Melrose, together with
refinements involving conormal regularity at submanifolds.  We
therefore begin by recalling notation and basic results about these objects.

Accordingly, the following preliminaries are essentially taken from \cite{radfielddecay}.
For a more thorough discussion of
b-pseudodifferential operators and b-geometry, we refer the reader to
Chapter 4 of Melrose~\cite{Melrose:APS}. 

In this section and the following, we initially take $M$ to be a
  manifold with boundary with coordinates $(\rho,y)\in [0,1)\times
  \RR^{n-1}$ yielding a product decomposition $M \supset U \sim [0,1)
  \times \pa M$ of a collar neighborhood of $\pa M.$ In particular,
  for now we lump the $v$ variable in with the other boundary
  variables as it will not play a distinguished role.

The space of \emph{b-vector fields}, denoted $\mathcal{V}_{b}(M),$ is the
vector space of vector fields on $M$
tangent to $\pd M$.  In local coordinates $(\rho, y)$ near $\pd M$,
they are spanned over $C^{\infty}(M)$ by the vector fields
$\rho\pd[\rho]$ and $\pd[y]$. We note that $\rho\pd[\rho]$ is
well-defined, independent of choices of coordinates, modulo $\rho \mathcal{V}_{b}(M)$; one may call this the
{\em b-normal vector field} to the boundary. One easily verifies that $\mathcal{V}_{b}(M)$ forms a Lie
algebra. The set of b-differential operators, $\Diffb^{*}(M)$, is the
universal enveloping algebra of this Lie algebra:
it is the filtered algebra consisting of operators of the form
\begin{equation}\label{exampleboperator}
A=\sum_{\smallabs{\alpha}+j\leq m} a_{j,\alpha}(\rho,y) (\rho D_\rho)^j
D_y^\alpha \in \Diffb^m(M)
\end{equation}
 (locally near $\pa M$) with the coefficients $a_{j,\alpha} \in \CI(M).$
We further define a bi-filtered algebra by setting
$$
\Diff_{\bl}^{m,l}(M)\equiv\rho^{-l} \Diff_{\bl}^m(M).
$$
The first index (here $m$) is the order of the operator and the second
(here $l$) is the weight.

The
b-pseudodifferential operators $\Psib^{*}(M)$ are the ``quantization''
of this Lie algebra, formally consisting of operators of the form
$$
b(\rho,y,\rho D_\rho, D_y)
$$
with $b(\rho,y,\xi,\eta)$ a Kohn--Nirenberg symbol (i.e., a symbol
smooth in all variables with an asymptotic expansion in decreasing
powers of $(\xi^{2} + |\eta|^{2})^{1/2}$); likewise we let
$$
\Psib^{m,l}(M)=\rho^{-l}\Psib^m(M)
$$
and obtain a bi-graded algebra.

The space $\mathcal{V}_{b}(M)$ is in fact the space of sections of a
smooth vector bundle over $M,$ the \emph{b-tangent bundle}, denoted
$\Tb M.$ The sections of this bundle are of course locally spanned by
the vector fields $\rho\pa_\rho,\pa_y.$ The dual bundle to $\Tb M$ is
denoted $\Tbstar M$ and has sections locally spanned over $\CI(M)$ by
the one-forms $d\rho/\rho, dy.$ We also employ the \emph{fiber
  compactification} $\overline{\Tb^{*}}M$ of $\Tbstar M$, in which we
radially compactify each fiber.  If we let
$$
\xi \frac{d\rho}{\rho}+ \eta \cdot dy
$$
denote the canonical one-form on $\Tbstar M$ then
a defining function for the boundary ``at
infinity'' of the fiber-compactification is
$$
\nu=(\xi^2+\abs{\eta}^2)^{-1/2};
$$
a (redundant) set of local coordinates on each
fiber of the compactification near $\{v= \rho = 0\}$ is given by
\begin{equation*}
  \nu,\ \hat{\xi} = \nu \xi,\  \hat{\eta} = \nu \eta.
\end{equation*}

The symbols of operators in $\Psib^*(M)$ are thus Kohn-Nirenberg
symbols defined on $\Tbstar M.$ The principal symbol map, denoted
$\sigma_{\bl},$ maps (the classical subalgebra of)
$\Psib^{m,l}(M)$ to $\rho^{-l}$ times
homogeneous functions of order $m$ on $\Tbstar M.$ In the particular
case of the subalgebra $\Diff_{\bl}^{m,l}(M),$ if $A$ is given by
\eqref{exampleboperator} we have
$$
\sigma_{\bl,m,l}(\rho^{-l} A)=\rho^{-l} \sum_{\smallabs{\alpha}+j\leq m} a_{j,\alpha}(\rho,y) \xi^j
\eta^\alpha
$$
where $\xi,\eta$ are ``canonical'' fiber coordinates on $\Tbstar M$
defined by specifying that the canonical one-form be 
$$
\xi \frac{d\rho}\rho + \eta \cdot dy.
$$

There is a canonical symplectic structure of $\Tbstar M^{\circ}$ given by the
exterior derivative of the canonical one-form
$$
\frac 1\rho d\xi \wedge d\rho + d\eta \wedge dy.
$$
The symbol of the commutator operators in $\Psib^* (M)$ is one order lower than
the product, with principal symbol given by the Poisson bracket of the
principal symbols with respect to this structure.  By contrast,
 the
\emph{weight} (second index) of the commutator is, in general, no better than that
of the product, owing to noncommutativity of the normal operators
introduced below.

Here and throughout this paper
we fix a
``b-density,'' which is to say a density which near the boundary is of the
form
$$
f(\rho,y) \abs{\frac{d\rho}\rho\wedge dy_1 \wedge\dots \wedge dy_{n-1}}
$$
with $f>0$ everywhere and smooth down to $\rho = 0$.
Let $L^2_\bl(M)$ denote the space of square integrable functions with respect to
the b-density.
We let $\Hb^m(M)$ denote the
Sobolev space of order $m$ relative to $L^2_\bl(M)$ corresponding to
the algebras $\Diffb^m(M)$ and $\Psib^m(M)$.  In other words, for
$m\geq 0$, fixing $A\in\Psib^m(M)$ elliptic, one has $w\in\Hb^m(M)$ if
$w\in L^2_\bl(M)$ and $Aw\in L^2_\bl(M)$; this is independent of the
choice of the elliptic $A$.  For $m$ negative, the space is defined by
dualization.  (For $m$ a positive integer, one can alternatively give a
characterization in terms of $\Diffb^m(M)$.)  Let
$\Hb^{m,l}(M)=\rho^{l}\Hb^m(M)$ denote the corresponding weighted
spaces.  The space $\Hb^{\infty,l}(M)$ are of special importance, as
they are the spaces of \emph{conormal distributions} with respect to
the boundary (having different possible boundary weights).  They can
more be easily characterized without any microlocal methods by the
iterated regularity condition
$$
u \in \Hb^{\infty,l}(M) \Longleftrightarrow V_1,\dots V_N u \in \rho^l L^2(M)\  \forall N,\ \forall V_j \in \Vb(M).
$$

We recall also that associated to the algebra $\Psib^*(M)$ is
associated a notion of Sobolev wavefront set:
$\WFb^{m,l}(w)\subset \Sb^*M$ is defined only for $w\in \Hb^{-\infty,l}$
(since $\Psib(M)$ is \emph{not} commutative to leading order in
the decay filtration); the definition is then $\alpha\notin\WFb^{m,l}(w)$ if
there is $Q\in\Psib^{0,0}(M)$ elliptic at $\alpha$ such that $Qw\in
\Hb^{m,l}(M)$, or equivalently if there is $Q'\in\Psib^{m,l}(M)$
elliptic at $\alpha$ such that $Q'w\in L^2_{\bl}(M)$. We refer to
\cite[Section~18.3]{Hormander:v3} for a discussion of $\WFb$ from a
more classical perspective, and \cite[Section~3]{Melrose-Vasy-Wunsch:Propagation}
for a general description of the wave front set in the setting of various pseudodifferential
algebras; \cite[Sections~2 and 3]{Vasy:corners} provide
another discussion, including on the b-wave front set relative to
spaces other than $L^2_{\bl}(M)$.

In addition to the principal symbol, which specifies high-frequency
asymptotics of an operator, we will employ the ``normal operator'' which
measures the boundary asymptotics.  For a b-differential operator given by
\eqref{exampleboperator}, this is simply the dilation-invariant operator
given by freezing the coefficients of $\rho D_\rho$ and $D_y$ at $\rho=0,$ hence
\begin{equation}\label{normop}
N(A)\equiv \sum_{\smallabs{\alpha}+j\leq m} a_{j,\alpha}(0,y) (\rho D_\rho)^j
D_y^\alpha \in \Diffb^m([0,\infty) \times \pa M).
\end{equation}

It is instructive in studying operators that are approximately
dilation-invariant near $\pa M$ to employ the Mellin transform.  Thus
 we define the Mellin transform
of $u$ a distribution on $M$ (suitably localized near the boundary) by
setting
\begin{equation}\label{Mellin}
\Mellin u(\sigma, y) = \int \chi(\rho) \rho^{-\imath\sigma -1} \, d\rho
\end{equation}
where $\chi$ is compactly supported and equal to $1$ near $0.$

The Mellin conjugate of the operator $N(A)$ is known as the ``reduced normal
operator'' and, if $N(A)$ is given by \eqref{normop}, the reduced
normal operator is simply the family (in $\sigma$) of operators on $\pa M$
given by
\begin{equation}\label{rednormal}
\widehat{N}(A)\equiv \sum_{\smallabs{\alpha}+j\leq m} a_{j,\alpha}(0,y) \sigma^j
D_y^\alpha.
\end{equation}
This construction can be extended to b-pseudodifferential operators, but 
we will only require it in the differential setting here.  Moreover,
while the construction is more subtle if we extend our coefficient
ring to $\CIlog,$ as we would need to do to consider the d'Alembertian
following the logarithmic coordinate change we will employ below, we
will in practice only employ this construction in the setting of our
original manifold with its smooth coordinates.

The Mellin transform is a useful tool in studying asymptotic
expansions in powers of $\rho$ (and $\log \rho$).  In particular, we
recall from Section 5.10 of \cite{Melrose:APS} that if $u$ is a
distribution on our manifold with boundary $M,$ we write
$$
u \in \phg{E}(M)
$$
iff $u$ is conormal to $\pa M$ with
$$
u \sim \sum_{(z,k) \in E} \rho^{\imath z} (\log \rho)^k a_{z,k}
$$
where $a_{z,k}$ are smooth coefficients on $y \in \pa M.$  Here $E$ is
an \emph{index set,} which is required to satisfy the following
properties\footnote{We have chosen to use the index set conventions of
  \cite{DAOMWC} rather than those in \cite{Melrose:APS}, which differ
  by a factor of $\imath$ in how the powers in the expansion relate to
the $z$ variable in the index set.}:
\begin{itemize}
\item $E \subset \CC \times \{0, 1, 2,\dots\}.$
\item $E$ is discrete.
\item $(z_j, k_j) \in E$ and $\abs{(z_j, k_j)} \to \infty
  \Longrightarrow$ $\Im z_j \to -\infty.$
\item $(z,k) \in E \Longrightarrow (z,\ell) \in E$ for $\ell=0, \dots,
  k-1$ as well.
\item $(z,k) \in E \Longrightarrow$ $(z-j\imath,k) \in E$ for $j \in \NN.$
\end{itemize}
We refer the reader to \cite{Melrose:APS} for an account of why these
conditions are natural ones to impose.  When $z \in \CC$ denotes an
index set, this means the smallest index set containing $(z,0),$ i.e.,
$\{(z-j\imath , 0): j \in \NN\}.$

We now remark that we may characterize distributions in $\phg{E}(M)$ in
two different ways: by Mellin transform, or by applying radial vector
fields.

To see the former, we recall that by Proposition~5.27 of
\cite{Melrose:APS}, we have $u \in \phg{E}(M)$ iff its Mellin transform
$\Mellin u$ is meromorphic, with poles of order $k$ only at points
$z$ such that $(z,k-1) \in E,$ as well as satisfying appropriate
decay estimates in $\sigma.$  (We will state a quantitative $L^2$
version of this result below, hence will not discuss the estimates
here.)

Alternatively, we recall that we may test for polyhomogeneity by use
of \emph{radial vector fields}:  Let $R$ denote the vector field $\rho
D_\rho$ (recalling that $D_\rho$ has a factor of $\imath^{-1}$ built
into it).  We can characterize $u \in \phg{E}(M)$ for $E =
\{(z_j,k_j)\}$ by the requirement that for all $l$ there exists
$\gamma_l$ with $\gamma_l \to \infty$ as $l \to +\infty$ such
that
\begin{equation}
  \label{eq:scalingtest}
  \left( \prod_{(z,k) \in E,\ \im z>-l}(R- z)\right)u
  \in \Hb^{\infty, \gamma_{l}}(M).
\end{equation}
(Note that by our index set conventions, the product includes $k+1$
factors of $(R-z_j)$ if $(z_j,k) \in E,$ since $(z_j, 0), \dots (z_j,
k-1)$ are in $E$ as well.)

Theorem~\ref{mainthm} is about polyhomogeneity not just to
one but to \emph{two} boundary hypersurfaces of a manifold with
codimension-two corners given by blow-up of our original spacetime $M$
at $S_+.$  We thus make a few remarks here on the generalization of
the theory of polyhomogeneity to this context; it is covered in
some detail in Section~5.10 of \cite{Melrose:APS}, but that
treatment only deals with the case where all but one of the index sets
are the set $$0 \equiv \{(-j\imath,0): j \in \NN\}$$ of indices for
smooth functions.  The more general case is treated in the unpublished \cite{DAOMWC}, but
follows similarly.  Thus, here we have an index set at a codimension
two corner with defining functions $\rho_1,\rho_2$ such that
$E=(E_1,E_2)$ with $E_j$ an index set at each of the boundary
hypersurfaces individually.  The idea is simply that $u$ has an
expansion at each boundary hypersurface with coefficients that are polyhomogeneous
at the other:
$$
u \in \phg{E}(M)
$$
iff for each $\ell=1,2,$ we have
$$
u \sim \sum_{(z, k) \in E_\ell} \phi_\ell(z,k) \rho_\ell^{\imath z} (\log
\rho_\ell)^{k}\bmod H^{\infty,\gamma_{\ell}} (M),
$$
where for each $(z,k)$ we have coefficients
$$
\phi_\ell (z,k)  \in \phg{E(\ell)}
$$
with $E(\ell)$ given by $(0, E_2)$ resp.\ $(E_1,0)$ for
$\ell=1,2$ and where
for $\ell=1,2$ $\gamma_{\ell} = (\infty, -A)$ resp.\ $(-A, \infty)$ with
fixed $A >\sup \{\Im z: (z,k) \in E_\ell$, $\ell=1,2$\}. 

In testing for polyhomogeneity at two boundary hypersurfaces by radial
vector fields, it is of considerable importance that it suffices to
test \emph{individually} at each boundary hypersurface, with uniform
estimates at the other; this is a consequence of a characterization by
multiple Mellin transforms (see Chapter 4 of Melrose's
book~\cite{DAOMWC}, or indeed the Appendix of the PhD thesis of
Economakis~\cite{Economakis}, where a proof provided by Mazzeo is
presented).  Thus we will in particular use the following:
\begin{proposition}[Mazzeo, Melrose]\label{proposition:doublemellin}
Let $R_\ell$ denote $\rho_\ell D_{\rho_\ell},$ the radial vector
field at the $\ell$'th hypersurface.
Suppose that for each $\ell=1,2,$ there exists a $\gamma'$, and for all $A$ there is a $\gamma_{A}$, with
$\lim_{A\to+\infty}\gamma_{A} = \infty$ such that
\begin{equation}
  \label{eq:rad-vf-exp}
  \left( \prod_{(z,k) \in E_\ell,\ \im z>-A}(R_\ell-z)\right)u
  \in \Hb^{\infty, \gamma_{A},\gamma '},
\end{equation}
where $\gamma_{A}$ refers to the growth slot for the $\ell$'th hypersurface, and (abusing
notation) $\gamma'$ to the
growth at the other boundary hypersurfaces. 
Then $u \in \phg{E}(M)$ where $E=(E_1,E_2)$.
\end{proposition}
Note that there is no
requirement in \eqref{eq:rad-vf-exp} that the coefficients in the expansion (or
indeed the remainder on the right hand side) be polyhomogeneous; this
follows automatically when~\eqref{eq:rad-vf-exp} is required for all
boundary hypersurfaces $H.$

We let $\ressetinit$ denote the ``index
set''\footnote{Note that this is not technically an index set as
  defined, as it is not closed under shifts by $-\imath$.} of poles of
the operator family $P_{\sigma}^{-1}$ with imaginary part less than
some fixed $\varsigma_{0}$.  Here we have
\begin{equation*}
  P_{\sigma} \equiv \widehat{N}(L)(\sigma)
\end{equation*}
with $L$ the rescaled conjugate of $\Box_{g}$ given by \eqref{Ldef}
and $\widehat{N}(L)$ the reduced normal operator as defined in
\eqref{rednormal}.  (The spaces on which we should consider this
operator to act will be defined below, in Section~\ref{sec:Fredholm}.)
Thus $(\sigma_{0}, k) \in \ressetinit$ if $\sigma_{0}$ is a pole of
$P_{\sigma}^{-1}$ of order at least $k+1$.   Though $P_{\sigma}^{-1}$
may have poles $\sigma_{j}$ with $\Im \sigma_{j} \to +\infty$, the
presence of $\varsigma_{0}$ in the definition restricts our attention
to a lower half-plane.  In practice, we fix $\varsigma_{0}$ large
enough to consider only the half-plane in which our function is not a
priori holomorphic.

To account for accidental multiplicities arising from multiplication
by $\CI$ (or $\CIlog$, defined in Definition~\ref{def:CIlog})
functions, we must also include in the resonance set the shifts of
$\ressetinit$ corresponding to the index sets of $\CI$ (or $\CIlog$)
functions.  Namely, for each $j=1,2, \dots$, we set
$$
\mathcal{E}_j =
  \{ (\sigma -\imath j, k) : (\sigma , k) \in \ressetinit
  (\varsigma_{0})\} 
$$
We define the massless resonance set as the extended union of $\mathcal{E}_{j}$:
\begin{equation}
  \label{eq:resset}
  \masslessres \equiv \ressetinit \eu \mathcal{E}_{1}\eu
  \mathcal{E}_{2}\eu \dots,
\end{equation}
where $\eu$ denotes the extended union of index sets as
in~\cite[Section 5.18]{Melrose:APS}: 
$$
E \eu F \equiv E \cup F \cup \{(z,k): (z,\ell_1) \in E,\
(z,\ell_2) \in F,\ k=\ell_1+\ell_2+1\};
$$
this corresponds to the increase in order of the poles of a product of
meromorphic function in the case when poles of the two functions
coincide.  Finally we define the resonance set that we obtain on
the ``logified'' space---when $m \neq 0$---by transformation of
$\masslessres$ (see Proposition~\ref{proposition:asymptotic.log}):
$$
\resset\equiv \begin{cases} \masslessres & m=0\\
\{ (\sigma - j\imath , k + \ell) : (\sigma, k) \in \masslessres,\ 0
     \leq \ell \leq j\} & \mass \neq 0
\end{cases}
$$

We also set 
$$
\E_{\scri}=\begin{cases}
0, & m=0\\
\{(-j \imath, \ell),\ \ell\in \{0, \dots, 2j\}, & m \neq 0.
\end{cases}
$$
Thus, the latter is the index set describing an expansion in $\rho^j
(\log \rho)^\ell$ for $\ell =0,\dots, 2j.$

Finally, write the ``total index set'' as 
$$
\Etot=(\resset, \E_{\scri})
$$
where the two index sets on the RHS are for the lift of $C_+$ and
$\scri^+$ respectively in the radiation field blowup.  This is the
index set that we will show occurs in the asymptotic expansion.

In an intermediate step of our construction, we will
need to consider slightly different kinds of asymptotic expansions: those that
are global expansions in the $\rho$ variable of $M$ but with
coefficients that are not smooth: they will have conormal
singularities of increasing orders at $S_+.$  Index sets and
manipulations for these expansions are defined analogously to those of
the smooth expansions described above, but we need to slightly clarify
the testing definition:
Suppose that 
\begin{equation}\label{asymptotic.conormal}
u \sim  \sum a_j(v,y) \rho^{\imath \sigma_j} (\log \rho)^{k_j}
\end{equation}
where now we assume that overall, $u$ is a conormal distribution
with respect to
$\rho=v=0,$ and that for some fixed $q_0, s_0,L$
$$
a_j \in I^{(q_0-\Re (\imath\sigma_j))}
$$
are also conormal, and where the asymptotic sum now means that
$$
u - \sum_{\Im \sigma_j \geq -A} a_j(v,y) \rho^{\imath \sigma_j} (\log
\rho)^{k_j} \in \rho^{L+A-0}
\Hb^{s_0 -A}(M).
$$
Thus the the remainder has better decay at the cost of conormal
regularity (and since $u$ is a priori conormal w.r.t.\ $N^* S_+,$ this loss of
regularity is only there).
\begin{proposition}\label{proposition:asymptotic.conormal}
A distribution $u$ conormal with respect to $N^* S_+$ enjoys the expansion
\eqref{asymptotic.conormal} (interpreted as described above) if and
only if we have
$$
\prod_{\Im \sigma_j>-A} (\rho D_\rho-\sigma_j)^{k_j} u \in \rho^{L+A-0}
\Hb^{q_0-A}(M).
$$
\end{proposition}

\subsection{Scattering geometry}

In addition to the notion of b-geometry, we also need to study a
different set of metric and operator structures on a manifold with
boundary.  If we radially compactify Euclidean space, we remark that
linear vector fields become the b-vector fields described
above, while \emph{constant coefficient} vector fields become elements
of 
$$
\Vsc (M) \equiv \rho \Vb (M).
$$
These ``scattering'' vector fields are thus spanned over $\CI(M)$ by
$\rho^2 \pa_\rho$ and $\rho \pa_y$ in the coordinates employed above,
and as with b-vector fields, they are sections of a bundle, denoted
$\Tsc M.$  The dual bundle, $\Tscstar M,$ has sections spanned by the
one-forms
$$
\frac{d \rho}{\rho^2},\ \frac{dy}{\rho}.
$$
The Euclidean and Minkowski metrics, under radial compactification of
Euclidean resp.\ Minkowski spaces to a ball, are quadratic in these
one-forms, and hence (non-degenerate) quadratic forms on $\Tsc M.$

We may build ``scattering differential operators'' out of scattering
vector fields by setting
$$
A=\sum_{\smallabs{\alpha}+j\leq m} a_{j,\alpha}(\rho,y) (\rho^2 D_\rho)^j
(\rho D_y)^\alpha \in \Diffsc^m(M)
$$
There is a well defined ``scattering principal symbol''
$\sigmasc^k(P)$ which replaces $\rho^2 D_\rho$ resp.\ $\rho D_y$ by
$\xisc$ resp.\ $\etasc,$ their canonical dual variables in the fibers of
$\Tscstar M.$  See \cite{RBMSpec} for details, as well as the
construction of the associated pseudo\-dif\-ferential calculus.

\section{Long-range scattering geometry}
\label{sec:long-range-scatt}

In this section, we specify our geometric hypotheses in detail.

Let $(M,g)$ be an $n$-dimensional manifold with boundary $X = \pa M$
equipped with a Lorentzian metric $g$ over $M^{\circ}$ such that $g$
extends to be a nondegenerate quadratic form on $\Tsc M$ of signature
$(+, - , \dots, -)$.

We motivate our definition of Lorentzian scattering metrics by
recalling that if we radially compactify Minkowski space by setting
$t=\rho^{-1} \cos \theta,$ $x_j=\rho^{-1}\omega_j 
\sin \theta$ with $\omega\in S^{n-2}$ and then set $v=\cos 2\theta,$
the metric becomes:
\begin{equation}\label{Minkowskimetric}
  g = v \frac{d\rho^{2}}{\rho^{4}} - \frac{v}{4(1-v^{2})} \frac{dv^{2}}{\rho^{2}} -
  \frac{1}{2} \big( \frac{d\rho}{\rho^{2}}\otimes \frac{dv}{\rho} +
  \frac{dv}{\rho} \otimes \frac{d\rho}{\rho^{2}}\big) - \frac{1-v}{2}\frac{d\omega^{2}}{\rho^{2}}.
\end{equation} 
This motivates the form of the following definition when $m=0.$ The
more general case is of course motivated by the need to include the
case of nontrivial solution to the Einstein vacuum equations, and in
particular we show below in Appendix~\ref{appendix:Kerr} that the Kerr
solution of the Einstein equations is in this broader class
\emph{sufficiently far from the event horizon and away from timelike
  infinity}: for any $\ep>0$ and $C_\ep$ sufficiently large, the
region described in Boyer-Lindquist coordinates by $r>C_\ep+\ep t$
will have the desired metric form.

\begin{definition}
  \label{def:scmetric}
  We say that $g$ is a long-range Lorentzian scattering metric if $g$
  is a smooth, Lorentzian signature, symmetric bilinear form on $\Tsc
  M$, and there exist a boundary defining function $\rho$ for $M$, a
  function $v\in \CI(M)$, and a constant $\mass\in \reals$ so that
  \begin{enumerate}
  \item 
    when $V$ is a scattering normal vector field, $g(V,V)$ has the
    same sign as $v$ at $\rho = 0$, and
  \item in a neighborhood of $\{ v = 0, \rho = 0\}$ we have
    \begin{equation*}
      g = (v - \mass \rho) \frac{d\rho^{2}}{\rho^{4}} - \left(
        \frac{d\rho}{\rho^{2}}\otimes \frac{\vartheta}{\rho} +
        \frac{\vartheta}{\rho}\otimes \frac{d\rho}{\rho^{2}}\right) - \frac{\tilde{g}}{\rho^{2}}
    \end{equation*}
    with $\vartheta$ a smooth $1$-form on $M$ and $\tilde{g}$ a smooth
    symmetric $2$-cotensor on $M$ so that
    \begin{equation*}
      \tilde{g}\rvert_{\Ann(d\rho, dv)} \text{ is positive definite.}
    \end{equation*}
    We further require that
    \begin{equation*}
      \vartheta = \frac{1}{2}dv + O(v) + O(\rho) \text{ near } \rho = v = 0.
    \end{equation*}
  \end{enumerate}
\end{definition}

We make two additional global assumptions on the structure of our spacetime:
\begin{definition}
  \label{def:non-trapping}
  A Lorentzian scattering metric is \emph{non-trapping} if
  \begin{enumerate}
  \item The set $S = \{ v = 0, \rho = 0\}\subset X$ splits into
    $S_{+}$ and $S_{-}$, each a disjoint union of connected
    components; we further assume that $\{ v > 0 \}\subset X$ splits
    into components $C_{\pm}$ with $S_{\pm} = \pa C_{\pm}$.  We denote
    by $C_{0}$ the subset of $X$ where $v < 0$.
  \item The projections of all null bicharacteristics on $\Tscstar M
    \setminus o$ tend to $S_{\pm}$ as their parameter tends to $\pm
    \infty$ (or vice versa).
  \end{enumerate}
\end{definition}

In particular, our non-trapping hypothesis guarantees that $(M,g)$ is
time-orientable: At each point the light cone has two components; we
specify that the future-directed component is the one from which the
nullbicharacteristics for the forward Hamilton flow tend to $S_{+}$.

The final definition needed to make sense of the statement of
Theorem~\ref{mainthm} is the following.
\begin{definition}
  Let $\Box_g u=f$ on $(M,g)$ a Lorentzian scattering manifold.  We
  say that $u$ is a \emph{forward solution} if $u$ is smooth near
  $\overline{C_{-}}$ and vanishes to infinite order there.
\end{definition}

We now analyze the inverse metric.  Our metric, as a metric on the
fibers of $\Tsc_{X}M$, i.e., in the frame
\begin{equation*}
  \rho^{2}\pa_{\rho}, \, \rho \pa_{v}, \, \rho \pa _{y}
\end{equation*}
has the block form
\begin{equation}
  \label{eq:metric-at-bdry}
  G_{0} =
  \begin{pmatrix}
    v & -\frac{1}{2} + a_{0}v & a_{1}v & \dots & a_{n-2}v \\
    -\frac{1}{2} + a_{0}v & b & c_{1} & \dots & c_{n-2} \\
    a_{1}v & c_{1} & -h_{1,1} & \dots & -h_{n-2,1} \\
    \vdots & \vdots & \vdots & \ddots & \vdots \\
    a_{n-2}v & c_{n-2} & -h_{1,n-2} & \dots & -h_{n-2,n-2}
  \end{pmatrix},
\end{equation}
with the lower $(n-2)\times (n-2)$ block negative definite, hence
$h_{ij}$ is positive definite.  Blockwise inversion shows that in the
frame 
\begin{equation*}
  \frac{d\rho}{\rho^{2}},\, \frac{dv}{\rho}, \, \frac{dy}{\rho},
\end{equation*}
the inverse metric when restricted to the boundary has the block
form
\begin{equation*}
  G_{0}^{-1} =
  \begin{pmatrix}
    \omega & -2 + \alpha v & -\frac{1}{2}\mu^{\mathrm{T}} + O(v)\\
    -2 + \alpha v & -4 v + \beta v^{2} & -v
    \Upsilon^{\mathrm{T}} + O(v^{2})\\
    -\frac{1}{2}\mu + O(v) & -v\Upsilon + O(v^{2}) & -h^{-1} + O(v)
  \end{pmatrix}.
\end{equation*}
In the above, $h^{-1} = h^{ij}$ is the inverse matrix of $h_{ij}$,
while $\omega, \alpha, \beta, \mu_{j}$, and $\Upsilon_{j}$ are smooth near $\rho = v
= 0$, and $A^{\mathrm{T}}$ denotes the transpose of a matrix $A$.

In a neighborhood of the boundary, i.e., at $\rho \neq 0$, there are
further correction terms in the inverse metric as the actual metric is
given by
\begin{align*}
  G&= G_{0} + H, \\
  H &=
  \begin{pmatrix}
    -\mass \rho + O(\rho^{2}) & O(\rho) & O(\rho) \\
    O(\rho) & O(\rho) & O(\rho) \\
    O(\rho) & O(\rho) & O(\rho)
  \end{pmatrix}.
\end{align*}
Thus in the inverse frame above,
\begin{equation}
  \label{eq:inverse-perturbation}
  G^{-1} = G_{0}^{-1} +
  \begin{pmatrix}
    O(\rho) & O(\rho) & O(\rho) \\
    O(\rho) & 4 \mass \rho+O(\rho^2) + O(\rho v) & O(\rho) \\
    O(\rho) & O(\rho) & O(\rho) 
  \end{pmatrix}.
\end{equation}

Thus in the coordinate frame $\pd[\rho]$, $\pd[v]$, $\pd[y]$, the dual
metric becomes
\begin{equation}
  \label{eq:dualmetric}
  \begin{pmatrix}
    g^{\rho\rho}\rho^{4}  + O(\rho^{5}) & g^{\rho v} \rho^{3} +
    O(\rho^{4}) & g^{\rho y} \rho^{3} + O(\rho^{4}) \\
    g^{\rho v} \rho^{3} + O(\rho^{4}) & g^{vv}\rho^{2}+ O(\rho^{4}) +
    O(\rho^{3}v) & g^{v y}\rho^{2} + O(\rho^{3}) \\
    g^{\rho y}\rho^{3} + O(\rho^{4}) & g^{vy}\rho^{2} + O(\rho^{3}) &
    g^{yy}\rho^{2} + O(\rho^{3})
  \end{pmatrix},
\end{equation}
where $g^{\bullet\bullet}$ are given by
\begin{align}
  \label{eq:inversemetriccomponents}
  g^{\rho\rho} &= \omega & g^{\rho v} &= -2 + \alpha v & g^{\rho y} &= -
  \frac{1}{2}\mu + O(v) \\
  g^{vv} &= - 4v + 4\mass \rho + \beta v^{2} & g^{vy} &= -v\Upsilon +
  O(v^{2}) & g^{yy} &= -h^{-1} + O(v) \notag
\end{align}
Again all terms are smooth.  We remark at this juncture that the
appearance of $m$ only at level of $O(\rho)$ terms means that the
normal operator of rescaled $\Box$ will be independent of $m,$ and
arguments involving only the inversion of this normal operator will
thus be identical to those in \cite{radfielddecay}.  Arguments
involving the detailed structure of $\Box$ near $S_+,$ however,
require serious modifications.

From \eqref{eq:dualmetric} it is easy to read off the scattering principal symbol of
$\Box_g:$ if the canonical one-form on $\Tscstar M$ is given by
$$
\xisc \frac{d\rho}{\rho^2}+\gammasc \frac{d v}{\rho}+\etasc \frac{d y}{\rho},
$$
then
\begin{multline}
\sigmasc^2 (\Box_g) = (\omega-m\rho+O(\rho^2)) \xisc^2 + (-4+2\alpha
v+O(\rho)) \xisc \gammasc+ (-4v+\beta v^2) \gammasc^2\\
-\big(h^{ij}+O(v)+O(\rho)\big) (\etasc)_i (\etasc)_j+
(-2v\Upsilon+O(v^2)+O(\rho))\gammasc \etasc+(-\mu+O(v)+O(\rho)) \xisc \etasc.
\end{multline}
The transition to the b-principal symbol is likewise quite simple,
since dividing by $\rho^2$ simply converts all sc-vector fields into
corresponding b-vector fields, with commutator terms contributing only
at lower order.  Hence we simply obtain
\begin{multline}\label{bsymbol1}
\sigmab^2 (\rho^{-2} \Box_g) = (\omega-m\rho+O(\rho^2)) \xi^2 + (-4+2\alpha
v+O(\rho)) \xi \gamma+ (-4v+\beta v^2) \gamma^2\\
-\big(h^{ij}+O(v)+O(\rho)\big) \eta_i \eta_j+
(-2v\Upsilon+O(v^2)+O(\rho))\gamma \eta+(-\mu +O(v)+O(\rho)) \xi \eta.
\end{multline}

\section{The Hamilton vector field and its radial set}
\label{sec:form-hvf}

We record the form of the b-Hamilton vector field of the conjugated
operator.  If $\lambda$ is the b-principal symbol of the conjugated
and rescaled operator
\begin{equation}
  \label{eq:Ldef}
  L = \rho^{-(n-2)/2 - 2}\Box_{g} \rho ^{(n-2)/2}
\end{equation}
then, since conjugation does not affect the principal symbol, we
still have, by \eqref{bsymbol1}
\begin{align}
  \label{eq:b-princ-symb}
  \lambda \equiv \sigmab (L) &= g^{\rho\rho}\xi^{2} - 2(2-\alpha v +
  O(\rho))\xi\gamma - (4v - 4\mass \rho - \beta v^{2} + O(\rho v) +
  O(\rho^{2}))\gamma^{2} \notag \\
  &\quad + 2g^{\rho y} \cdot \eta \xi - 2 (v \Upsilon + O(\rho))\cdot
  \eta \gamma + g^{y_{i}y_{j}}\eta_{i}\eta_{j}, 
\end{align}
which yields the Hamilton vector field
\begin{align*}
  \sH_{\lambda} &= \left( 2 g^{\rho\rho}\xi - 2 ( 2-\alpha v +
    O(\rho)) \gamma + 2g^{\rho y}\cdot \eta\right)\rho \pd[\rho] \\
  &\quad - 2 \left[ \left( 4v - \beta v^{2} - 4 \mass \rho + O(\rho v)
      + O(\rho^{2})\right)\gamma + (2- \alpha v + O(\rho))\xi +
    (v\Upsilon + O(\rho)) \cdot \eta \right]\pd[v]\\
  &\quad + 2\left( g^{\rho y}\xi - (v\Upsilon +O(\rho))\gamma +
    g^{yy}\eta\right)\cdot \pd[y] - (\rho \pd[\rho]\lambda)\pd[\xi] -
  (\pd[v]\lambda)\pd[\gamma] - (\pd[y]\lambda)\cdot\pd[\eta] .
\end{align*}

We now analyze the \emph{radial set} $\cR$ of the Hamilton vector
field within the characteristic set of $L$.  This is defined as the
conic set
$$
\cR=\{p \in \Tbstar M: \lambda(p)=0, \sH_{\lambda}\rvert_p \in \RR R\},
$$
where $R$ denotes the scaling vector field in the fibers of $\Tbstar
M.$ In order for $\sH_{\lambda}\rvert_{p}\in \RR R$, the projection
$\pi\sH_{\lambda}$ of $\sH_{\lambda}$ to the base must vanish as a
smooth vector field.  We recall that $\lambda$ is a nondegenerate
Lorentzian metric on the fibers of $\Tbstar M$ and denote the induced
b-metric on $\Tb M$ by $g_{b}$.  For a point $p = (x^{i},\zeta_{i})\in
\Tbstar M$ not in the zero section, the projection $\pi \sH_{\lambda}$
is given by
\begin{equation*}
  \pi \sH_{\lambda} = 2 \left( g^{\rho j}\zeta _{j}\rho \pd[\rho] +
    g^{v j}\zeta _{j} \pd[v] + g^{yj}\zeta _{j} \pd y\right).
\end{equation*}
In other words, at a point $p = (x, \zeta) \in \Tbstar M$, the
projection $\pi \sH_{\lambda}$ is the vector at $x$ associated to
$\zeta$ by regarding $g_{b}$ as a linear map $\Tbstar_{x} M \to
\Tb_{x} M$.  Thus $\pi \sH_{\lambda}$ must be a non-vanishing b-vector
field.  In particular, for it to vanish as a smooth vector field, it
must be a nonzero multiple of $\rho \pd[\rho]$.  We further have that 
\begin{equation*}
  g_{b} \left( \pi \sH_{\lambda}, \pi
    \sH_{\lambda}\right) = 4g_{ij}(g^{ik}\zeta_{k})(g^{j\ell}\zeta_{\ell}) =
    4(g^{j\ell}\zeta_{j}\zeta_{\ell}) = 4\lambda (p),
\end{equation*}
and so $\rho \pd[\rho]$ must be a null vector field at $\rho = 0$ and
thus $v=0$. An examination of the coefficients of the spatial
vector fields then shows that the radial set $\cR$ within $\rho = 0$
is exactly $v = 0$, $\eta = 0$, $\xi = 0$.  Equivalently (and this
will be used below---cf.\ \eqref{generators}), we can take it to be
defined by $\lambda, \rho, \eta, \xi$, substituting $\lambda$ for $v$
as a defining function.

On the fiber compactification of $\Tbstar M$ near $\cR$, we use local
coordinates
\begin{equation*}
  \nu = \frac{1}{\gamma}, \hat{\xi} = \frac{\xi}{\gamma}, \hat{\eta} = \frac{\eta}{\gamma},
\end{equation*}
and compute the linearization of $\sH_{\lambda}$ at $\cR$.  Modulo
terms vanishing quadratically at $\pd \cR$, we have
\begin{align*}
  \nu \sH_{\lambda} &= -4\rho \pd[\rho] + (-8v - 4\hat{\xi} + 8 \mass
  \rho)\pd[v] + 2 \left( g^{\rho y}\hat{\xi} - v \Upsilon + c\rho +
    g^{yy}\hat{\eta}\right)\pd[y] \\
  &\quad - 4 \mass \rho \pd[\hat{\xi}] - 4 \left( \nu \pd[\nu] +
    \hat{\xi} \pd[\hat{\xi}] + \hat{\eta}\pd[\hat{\eta}]\right) +
  \cI^{2}\Vf (\overline{\Tbstar} M),
\end{align*}
with $c$ smooth.

In particular, the linearization of $\nu \sH_{\lambda}$ has
eigenvectors and eigenvalues given by
\begin{align*}
  &dv + d\hat{\xi} - \mass\,d\rho \text{ with eigenvalue }-8, \\
  &d\rho , d\nu, d\hat{\eta}, \text{ with eigenvalue }-4, \\
  &4dy + 2 g^{yy}d\hat{\eta} + (2c - 3m\Upsilon - 2 m g^{\rho y})d\rho
  - \Upsilon dv + (2g^{\rho y} + \Upsilon)d\hat{\xi}, \text{ with eigenvalue } 0.
\end{align*}
For $\mass \neq 0$, this leaves one dimension unaccounted for, and in
a notable difference with the short-range case, for $\mass \neq 0$,
there is in fact a nontrivial Jordan block in the generalized
eigenspace $-4$, spanned by $d\rho$ and $d\hat{\xi}$.

Consequently, we must revisit the proof of propagation of b-regularity
to radial points (Proposition~4.4 of \cite{radfielddecay}) in this
context.  We undertake this in the following section.

\section{Propagation of b-regularity and module regularity}
\label{sec:prop-b-regul}

\begin{definition}
  \label{def:module}
  Let $\cM\subset \Psib^{1}(M)$ denote the $\Psib^{0}(M)$-module of
  pseudodifferential operators with principal symbol vanishing on the
  radial set $\cR = \{\rho = 0,\ v=0,\ \xi=0,\ \eta=0\}$.  We also let
  $\cMdiff\subset \Diff^{1}(M)$ denote the module of differential
  operators with principal symbol vanishing on the radial set $\cR$.  
\end{definition}

Note that a set of generators for $\cM$ over $\Psib^{0}(M)$ is given
by the vector fields $\rho \pd[\rho]$, $\rho\pd[v]$, $v\pd[v]$,
$\pd[y]$,
and $I$.  The differential module $\cMdiff$ is generated by the same
vector fields over $\CI (M)$.

We recall from~\cite{radfielddecay} that the module $\cM$ is closed
under commutators.  

If we disregard factors in $\cM^{2}$ we note that the operator $L$ defined
by equation~\eqref{eq:Ldef} takes a particularly simple form:
\begin{lemma}
  \label{lemma:LmodM}
  \begin{equation}
    \label{eq:Lmodule}
    L = 4\pd[v]\left( \rho \pd[\rho] + v \pd[v]\right) - 4 \mass \rho
      \pd[v]^{2} + \cM^{2}.
  \end{equation}
\end{lemma}
\begin{proof}
  As in the previous section, we let $g_{b}$ denote the induced
  b-metric given by $\lambda,$ so that $g_{b} = \rho^{2}g$.  We
  observe that $L$ and $\Box_{g_{b}}$ have the same principal symbol
  and are both self-adjoint with respect to the volume $\rho^n
  \sqrt{g},$ hence these operators agree up to a smooth zero-th order
  term (which is automatically in $\cM^{2}$).  We must thus show that
  $\Box_{g_{b}}$ has the desired form.

  To see this we start by noting that $\Box_{g_{b}}$ is an element of
  $\Diffb^{2}(M)$ and so the only terms of $\Box_{g_{b}}$ not lying in
  $\cM^{2}$ are those terms containing a $\pd[v]$ (because
  $\rho\pd[\rho]$ and $\pd[y]$ lie in $\cM$).  We then observe that
  \begin{align*}
    g_{b}^{vy} &= O(v), \quad g_{b}^{\rho v} = -2\rho + O(\rho v),\\
    g_{b}^{vv} &= -4v + 4\mass \rho + O(v^{2}) + O(\rho v) + O(\rho^{2}).
  \end{align*}
  Because $\sqrt{g_{b}} = \rho^{-2} A$, where $A$ is smooth and
  non-vanishing, it follows that $\Box_{g_{b}}$ (and hence $L$) has
  the desired form.
\end{proof}

We begin by recalling, just as in \cite{radfielddecay}, that
regularity/singularities of solutions to $L w=f \in \dCI(M)$
propagates along maximally extended integral curves of the Hamilton
vector field for a wide class of operators $L$:
let $L\in\Psib^{s,r}(M)$ be arbitrary, and let
$\Sigma\subset\Sb^*M$ denote the characteristic set of $L$, $\lambda$
denote the principal symbol of $L$ in $\Psib^{s,r}(M)$.
\begin{proposition}\label{proposition:bpropagation} Suppose $w \in \Hb^{-\infty,l}(M).$
Then
\begin{enumerate}\item
Elliptic regularity holds away from $\Sigma$, i.e.,
$$
\WFb^{m,l}(w)\subset\WFb^{m-s,l-r}(Lw)\cup\Sigma,
$$
\item In $\Sigma$, $\WFb^{m,l}(w)\setminus\WFb^{m-s+1,r-l}(Lw)$ is a
union of maximally extended bicharacteristics, i.e., integral curves
of $\sH_\lambda$.
\end{enumerate}
\end{proposition}
Note that the order in $\WFb^{m-s+1,r-l}(Lw)$ is shifted by $1$
relative to the elliptic estimates, corresponding to the usual
hyperbolic loss.  This arises naturally in the positive commutator
estimates used to prove such hyperbolic estimates: commutators in
$\Psib(M)$ are one order lower than products in the differentiability
sense (the first index), but not in the decay order (the second
index); hence the change in the first order relative to elliptic
estimates but not in the second.  We refer the reader to, e.g.,
\cite{Vasy:corners}, for a proof; the idea is essentially a version of
the usual real-principal type propagation argument by positive
commutators.

Proposition~\ref{proposition:bpropagation} by itself fails to give any
useful information exactly at $\cR,$ the radial set of $L$.   To
analyze the solutions at $\cR$ we require a considerably subtler
result that yields propagation into and out of the radial set
but which is sensitive to the order of Sobolev regularity under
study.  The statement below is thus \emph{only about the particular operator
$L$} under study here, as it depends in detail on the behavior near $\cR.$
Our result here has the same statement as Proposition~4.4 of
\cite{radfielddecay} but as noted above is complicated by the
existence of a nontrivial Jordan block in the linearization of the
Hamilton vector field about $\cR$ in the long-range case considered
here.

\begin{proposition}
  \label{prop:radial-b-estimate}
  Let $L= \rho^{-(n-2)/2-2}\Box_{g}\rho^{(n-2)/2}$.  If $w \in
  \Hb^{-\infty,l}(M)$ for some $l$, $Lw \in \Hb^{m-1, l}$, and $w \in
  \Hb^{m,l}$ on a punctured neighborhood $U\setminus \pd\cR$ of $\pa
  \cR$ in $\Sb^{*}M$ (i.e., $\WFb^{m,l}(w)\cap (U \setminus\pa\cR) =
  \emptyset$) then for $m' \leq m$ with $m'+l < 1/2$, $w\in
  \Hb^{m',l}(M)$ at $\pa\cR$ (i.e., $\WFb^{m',l}(w)\cap \pa\cR =
  \emptyset$) and for $N \in \NN$ with $m' + N \leq m$ and for $A \in
  \cM^{N}$, $Aw$ is in $\Hb^{m',l}(M)$ at $\pa \cR$ (i.e.,
  $\WFb^{m',l}(Aw)\cap \pa\cR = \emptyset$).
\end{proposition}

We sketch the proof, focusing on the differences from
\cite{radfielddecay}.

\begin{proof}
  First we show propagation of ordinary b-regularity up to the
  threshold regularity $m'$.  We inductively show that
  $\WFb^{\tilde{m},l}(w) \cap \pd \cR = \emptyset$ assuming that we
  already have shown $\WFb^{m'', l}(w) \cap \pd \cR = \emptyset$ with
  $m'' = \tilde{m} - 1/2$.  As $w\in \Hb^{m_{0}, l}(M)$ for some
  $m_{0}$, we start with $\tilde{m} = \min (m_{0} + 1/2, m')$ and then,
  increasing $\tilde{m}$ by an amount $\leq 1/2$ at each step, we may
  reach $\tilde{m} = m'$ in finitely many steps.  

  To do this, we set
  \begin{equation*}
    a = \rho ^{-r}\nu^{-s}\phi^{2},
  \end{equation*}
  where $\phi \geq 0$, $\phi \equiv 1$ near $\cR$ and $\supp \phi
  \subset U$.  Taking $r + s < 0$ and constraining the support of
  $\phi$ appropriately gives
  \begin{equation*}
    \nu \sH_{\lambda}a = - b^{2} + e,
  \end{equation*}
  with $b$ elliptic near $\cR$ and $e$ supported on $\supp d\phi$,
  which we choose to be away from $\cR$.  Choosing $A \in
  \Psib^{s,r}(M)$ with symbol $a$ then gives
  \begin{equation*}
    \imath [L, A] = - B^{*}B + E + F,
  \end{equation*}
  with $E \in \Psib^{s+1,r}(M)$ microsupported away from $\cR$, $B \in
  \Psib^{(s+1)/2, r/2}(M)$, and $F \in \Psib^{s,r}(M)$.  Hence we have
  an estimate
 \begin{equation}
   \label{eq:propagation1}
   \norm{Bw}^{2} \leq \left| \langle Ew, w\rangle\right| + \left|
     \langle Fw, w\rangle\right| + 2 \left| \langle Lw, Aw\rangle\right|
 \end{equation}
 when $w$ is a priori sufficiently regular.  Given $\tilde{m}, l$, we
 take $s = 2\tilde{m} - 1$ and $r = 2l$ so that $s + r < 0$ is
 satisfied.  As $F$ has order $\leq 2m''$, the inductive assumption
 gives a bound on $\langle Fw, w\rangle$.  A standard regularization
 argument to justify the pairing then proves the proposition for
 $N=0$.

 Now we turn to the general case, following the methods of Hassell,
 Melrose, and Vasy~\cite{Hassell-Melrose-Vasy:Spectral,
   Hassell-Melrose-Vasy:Microlocal} (cf.\ also the appendix
 of~\cite{Melrose-Vasy-Wunsch:Propagation}).  In particular, we follow
 the treatment of Section 6 of~\cite{Hassell-Melrose-Vasy:Spectral}, which covers the
 propagation of regularity under test modules into and out of radial points.
 
 Thus as generators of the module we may take quantizations of
 $\nu^{-1}g$, where $g$ runs over the set
 \begin{equation}\label{generators}
   \{ \hat{\eta}, \rho, \hat{\xi}, \nu^{2}\lambda\}.
 \end{equation}
 Recall that $d\hat{\eta},\ d\rho$ are eigenvectors of the
 linearization of $\nu \sH_{\lambda}$ with eigenvalue $-4$ while
 $d\hat{\xi}$ lies in the same generalized eigenspace.  

Now let $G_{0} = I$ and let $G_{1},
 \dots, G_{n-1}$ be given by quantizing $\nu^{-1}\hat{\eta}$ and
 $\nu^{-1}\rho$; let $G_{n}$ be the quantization of
 $\nu^{-1}\hat{\xi}$ and let $G_{n+1} = \Lambda L$.  Here $\Lambda\in
 \Psib^{-1}$ has symbol $\nu$ near $\cR$.  We employ the obvious
 multi-index notation for $G^{\alpha}$.

 Since $d\nu,\ d\hat{\eta},\ d\rho$ have equal eigenvalues for $1 \leq j \leq
 n-1$, we have
 \begin{equation*}
   \imath \Lambda [G_{i}, L] = \sum_{j=1}^{n+1}C_{ij}G_{j}+E_i,
 \end{equation*}
 where $E_i \in \Psib^{-\infty}(M)$ and for $i \leq n-1$,
 \begin{equation}
   \label{eq:vanishingsymbol}
   \sigma_{\bl, 0, 0}(C_{ij})|_{\cR} = 0.
 \end{equation}
 By contrast,
 \begin{equation*}
   \imath \Lambda [G_{n}, L] = \sum_{j}C_{n,j}G_{j}+E_{n},
 \end{equation*}
 with $E_{n} \in \Psib^{-\infty}(M)$ and 
 \begin{equation}
   \label{eq:notquitevanishingsymbol}
   \sigma_{\bl,0,0}(C_{n,j})|_{\cR} =0, \ j\neq n-1;
 \end{equation}
 \emph{the term $C_{n,(n-1)}$ will not enjoy this vanishing property,
   however.}

 We now inductively control regularity under $\cM^{N}$, with $N = 0$
 being the case established above.  In proving regularity of $w$ under
 $\cM^{N}$ given regularity under $\cM^{N-1}$, we recall that it
 suffices to consider the application of elements $G^{\alpha}$ with
 $|\alpha | = N$ and with $\alpha_{n+1} = 0$, since the presence of a
 single factor of $G_{n+1}$ in the correct slot renders $w$ residual.
 (We can arrange that factors of $G_{n+1}$ are always in the correct
 slot as the induction hypothesis allows us to bound the commutators.)
 
 We thus consider the \emph{system} of commutators
 \begin{equation*}
   \imath [L, W_{\alpha}],
 \end{equation*}
 with
 \begin{equation*}
   W_{\alpha} = \epsilon^{-\alpha_{n-1}} \Op (\sqrt{a})^{*}(G^{\alpha})^{*}(G^{\alpha})\Op(\sqrt{a}),
 \end{equation*}
 where $a$ is chosen as above, $\epsilon > 0$ is small (to be fixed
 later), and where we let $\alpha$ run over all values with
 \begin{equation*}
   \smallabs{\alpha}= N, \alpha_{n+1} = 0.
 \end{equation*}

 As before, since $s+r < 0$ (and if the support of $\phi$ is
 sufficiently small), we have
 \begin{align}
   \label{eq:Jordanblockcommutators}
   \imath[L, W_{\alpha}] &= -
   \epsilon^{-\alpha_{n-1}}B^{*}(G^{\alpha})^{*}(G^{\alpha})B \notag \\
   &+ \sum_{\beta }
   \epsilon^{-\alpha_{n-1}}\Op(\sqrt{a})^{*}\left((G^{\beta})^{*}C_{\alpha\beta}(G^{\alpha})
     + (G^{\alpha})^{*}C_{\alpha\beta}'(G^{\beta})\right)\Op(\sqrt{a}) 
   \\
   &+ E_{\alpha} + F_{\alpha} \notag
 \end{align}
 Here the terms involving the $G^{\beta}$ (and adjoint) arise from the
 commutators of $L$ with the $G^{\alpha}$ (and adjoint) factors; $B$
 is elliptic near $\cR$, as before.  $E_\alpha$ is microsupported away
 from $\cR$, and $F_\alpha$ has lower order.
 Crucially, the vanishing of the
 symbols of $C_{ij}$ on $\cR$ imply that
 \begin{equation*}
   \sigma_{\bl,0,0}(C_{\alpha\beta}) = 0, \ \sigma_{\bl, 0,
     0}(C'_{\alpha\beta})=0,\ \text{on}\ \cR \text{ unless }\beta =
   \alpha + \delta^{n-1} - \delta^{n},
 \end{equation*}
 where $\delta ^{j}$ is the multi-index with $\delta^{j}_{i} = 0$ for
 $i \neq j$, $\delta^{j}_{j} = 1$.  Now on pairing
 equation~\eqref{eq:Jordanblockcommutators} with $w$, we note that:
 \begin{itemize}
 \item Terms with $\beta_{n+1} \neq 0$ are trivially bounded because
   $Lw$ is residual.  
 \item Terms with $\smallabs{\beta}< N$ can be absorbed in the positive terms
   by the inductive hypothesis and Cauchy--Schwarz.
 \item Terms with $\smallabs{\beta} = \smallabs{\alpha}$ can likewise
   be absorbed in the main positive terms unless $\beta = \alpha +
   \delta^{n-1} - \delta^{n}$ by the vanishing of the symbol
   (shrinking supports if necessary).
 \item Terms with $\beta = \alpha + \delta^{n-1} - \delta^{n}$ can be
   likewise handled by Cauchy--Schwarz, as they come
   with a coefficient $\epsilon^{-\alpha_{n-1}}$, while the
   corresponding positive term has coefficient $\epsilon^{-\beta_{n-1}}
   = \epsilon^{-\alpha_{n-1} -1}$.  Hence for $\epsilon$ sufficiently
   small, these terms, too, may be controlled by the main commutator terms.
 \end{itemize}
\end{proof}

\section{Fredholm properties}\label{sec:Fredholm}

We now turn to the Fredholm properties of the operator family
$P_\sigma$ on variable-order Sobolev spaces, which we can deduce from
the propagation theorems above.  This argument is identical to that
used in \cite{radfielddecay}, again following the strategy first used
by the second author in \cite{Vasy-Dyatlov:Microlocal-Kerr}.

\begin{definition}
  \label{def:holo}
  Let $\CC_\nu$ denote the halfspace $\Im \sigma>-\nu$ and let
  $\hol(\CC_\nu)$ denote holomorphic functions on this space.  For a
  Fr{\'e}chet space $\mathcal{F}$, let $$\hol(\CC_\nu)\cap \langle
  \sigma\rangle^{-k} L^\infty L^{2}(\reals ;\mathcal{F})$$ denote the
  space of $g_\sigma$ holomorphic in $\sigma \in \CC_\nu$ taking values in $\mathcal{F}$
  such that each seminorm
  $$
  \int_{-\infty}^\infty \norm{g_{\mu+\imath\nu'}}_\mathcal{\bullet}^2 \langle \mu
  \rangle^{2k} \, d\mu
  $$
  is uniformly bounded in $\nu'>-\nu$.
\end{definition}
Note the choice of signs: as $\nu$ increases, the halfspace gets
larger.

We will further allow elements of $\hol (\CC_{\nu})$ to take values in
\emph{$\sigma$-dependent} Sobolev spaces, or rather Sobolev spaces
with $\sigma$-dependent norms.  In particular, we allow
values in the standard semiclassical Sobolev spaces $H^{m}_{h}$ on a
compact manifold (without boundary), with semiclassical parameter $h =
\langle\sigma\rangle^{-1}$. Recall (see \cite[Section 8.3]{Zworski}) that these are the standard Sobolev spaces
and up to the equivalence of norms, for $h$ in compact subsets of
$(0,\infty)$, the norm is just the standard $H^m$ norm, but
the norm is $h$-dependent: for non-negative integers $m$, in
coordinates $y_j$, locally the norm $\|g\|_{H^m_h}$ is equivalent to
$\sqrt{\sum_{\smallabs{\alpha}\leq m}
  \|(hD_{y_j})^\alpha g\|^2_{L^2}}$.

We let $P_{\sigma}$ be the normal operator for the conjugated operator
$L = \rho^{-(n-2)/2}\rho^{-2}\Box_{g}\rho ^{(n-2)/2}$.  Recall that
under our global assumptions, the characteristic set of $P_{\sigma}$
in $S^{*}X$ has two parts $\Sigma_{\pm}$ such that the integral curves
of the Hamilton flow in $\Sigma_{\pm}$ tend to $S_{\pm}$ as the
parameter tends to $+\infty$.  

We recall now from~\cite{radfielddecay} that the radial points of the
Hamilton vector field of $P_{\sigma}$ (an operator on $X = \pd M$)
occur when $v=0$.  Indeed, they occur at
\begin{equation}
  \label{eq:Psigma-rad-set}
  \Lambda^{\epsilon_{1}}_{\epsilon_{2}} = \{ v= 0 , \eta = 0,
  \epsilon_{2}\gamma > 0\} \cap \Sigma_{\epsilon_{1}} \subset
  T^{*}X, \quad \epsilon_{i} = \pm,
\end{equation}
so that the $\pm$ in the superscript distinguishes ``past'' from
``future'' null infinity, while that in the subscript separates the
intersections with the two components of the characteristic set.  The
past and future radial sets are denoted $\Lambda^{\pm} =
\Lambda^{\pm}_{+}\cup \Lambda ^{\pm}_{-}$.  

The operator family $P_{\sigma}$ is Fredholm on appropriate
variable-order Sobolev spaces, which we now recall.  Let
$\bar{s}^{\pm}(\sigma)$ denote the threshold Sobolev exponents at
$\Lambda^{\pm}$, i.e., at the future and past radial sets.
From~\cite{radfielddecay}, we recall that in fact
\begin{equation*}
  \bar{s}^{\pm}(\sigma) = \frac{1}{2} + \Im \sigma.
\end{equation*}
Now let $s_{\tow}$ be a function on $S^{*}X$ so that
\begin{enumerate}
\item $s_{\tow}$ is constant near $\Lambda^{\pm}$,
\item $s_{\tow}$ is decreasing along the $\sH_{p}$-flow on
  $\Sigma_{+}$ and increasing on $\Sigma_{-}$,
\item $s_{\tow}$ is less than the threshold exponents at
  $\Lambda^{+}$, towards which we propagate our estimates, i.e.,
  $s_{\tow}|_{\Lambda^{+}}< \bar{s}^{+}(\sigma)$, and
\item $s_{\tow}$ is greater than the threshold value at $\Lambda^{-}$,
  away from which we propagate our estimates.  
\end{enumerate}
We also require a function $s_{\away}^{*}$ on $S^{*}X$ satisfying the
above assumptions for $P_{\sigma}^{*}$ with $\bar{s}^{\pm,*}(\sigma) =
- \bar{s}^{\pm}(\sigma) + 1$, and may thus take $s^{*}_{\away} = -
s_{\tow} + 1$ so that
\begin{equation*}
  (H^{s_{\tow}})^{*} = H^{s_{\away}*-1}, \quad (H^{s_{\tow}-1})^{*} = H^{s_{\away}^{*}}.
\end{equation*}
As $\Im\sigma$ decreases, the constant value $s(S_{+})$
assumed by $s_{\tow}$ near $S_{+}$ must satisfy $s(S_{+}) <
\frac{1}{2} + \Im \sigma$.  Because we are ultimately interested in
functions that are identically zero near $S_{-}$, we may typically
choose $s_{\tow}$ and $s_{\away}$ so that they are constant on the
support of our functions.

For $U\in H^{s_{\tow}}$ near $\Lambda^{-}$, propagation of regularity
from $\Lambda^{-}$ to $\Lambda^{+}$ yields estimates of the form
\begin{equation*}
  \Norm[H^{s_{\tow}}]{U} \leq C \left(
    \Norm[H^{s_{\tow}-1}]{P_{\sigma}U} + \Norm[H^{-N}]{U}\right),
\end{equation*}
with similar estimates holding for $P_{\sigma}$.  We may thus obtain
Fredholm properties for $P_{\sigma}$ and $P_{\sigma}^{*}$ by changing
the spaces slightly.  We set
\begin{equation*}
  \cY^{s_{\tow}-1} = H^{s_{\tow}-1}, \quad \cX ^{s_{\tow}} = \{ U \in
  H^{s_{\tow}}: P_{\sigma}U \in \cY^{s_{\tow}-1}\}.
\end{equation*}
(Recall that the last statement in the definition of $\cX^{s_{\tow}}$
depends only on the principal symbol of $P_{\sigma}$, which is
independent of $\sigma$.)

The following proposition then holds for $P_{\sigma}$:
\begin{proposition}[\cite{radfielddecay}, Proposition 5.1]
  \label{prop:fredholm}
  The family of maps $P_{\sigma}$ enjoys the following properties:
  \begin{enumerate}
  \item $P_{\sigma}: \cX^{s_{\tow}}\to\cY^{s_{\tow}-1}$ and
    $P_{\sigma}^{*}: \cX^{s_{\away}^{*}}\to \cY^{s_{\away}^{*}-1}$ are
    Fredholm maps.
  \item $P_{\sigma}$ is a holomorphic Fredholm family on these spaces
    in
    \begin{equation*}
      \CC_{s_{+}, s_{-}} = \{ \sigma \in \CC \mid s_{+} <
      \bar{s}^{+}(\sigma), s_{-} > \bar{s}^{-}(\sigma)\},
    \end{equation*}
    with $s_{\tow}|_{\Lambda^{\pm}} = s_{\pm}$.  $P_{\sigma}^{*}$ is
    antiholomorphic in the same region.  
  \end{enumerate}
\end{proposition}

Non-trapping versions of the above estimates yield the following
proposition as well:
\begin{proposition}[\cite{radfielddecay}, Proposition 5.2]
  \label{prop:nontrappingestimates}
  If the non-trapping hypothesis holds, then
  \begin{enumerate}
  \item $P_{\sigma}^{-1}$ has finitely many poles in each strip $a <
    \Im \sigma < b$.
  \item For all $a,b$ there exists $V$ such that
    \begin{equation*}
      \Norm[\cY^{s_{\tow}-1}_{|\sigma|^{-1}}\to
      \cX^{s_{\tow}}_{|\sigma|^{-1}}]{P_{\sigma}^{-1}} \leq C \langle
      \Re \sigma\rangle^{-1}
    \end{equation*}
    for $a < \Im \sigma < b$ and $\smallabs{\Re \sigma} > C$.
  \end{enumerate}
\end{proposition}

Here the spaces with $|\sigma|^{-1}$ subscripts refer to the variable
order versions of the semiclassical Sobolev spaces.  

An inductive argument about the Jordan block structure of
$P_{\sigma}^{-1}$ and the Cauchy integral formula establish the
following lemma as well:
\begin{lemma}[\cite{radfielddecay}, Lemma 8.3]
  \label{lemma:laurent}
  Let $\sigma_{0}$ be a pole of order $k$ of the operator family
  \begin{equation*}
    P_{\sigma}^{-1} : \cY^{s_{\tow}-1} \to \cX^{s_{\tow}}
  \end{equation*}
  and let
  \begin{equation*}
    (\sigma - \sigma_{0})^{-k}A_{k} + (\sigma -
    \sigma_{0})^{-k+1}A_{k-1} + \dots + (\sigma -
    \sigma_{0})^{-1}A_{1} + A_{0}
  \end{equation*}
  denote the Laurent expansion near $\sigma_{0}$, with $A_{0}$ locally
  holomorphic.  If a function $f$ vanishes in a neighborhood of
  $\overline{C_{-}}$, then $A_{\ell}f$ is supported in
  $\overline{C_{+}}$ for $\ell = 1, \dots, k$.  
\end{lemma}

\section{Logification}
\label{sec:logification}

The long-range term in the metric induces a logarithmic divergence of
the light cones near infinity when compared to the short-range
setting.  There are several ways to compensate for this fact: for
example, one could introduce a logarithmic correction when blowing up
$S_{\pm}$.  This method, however, causes problems, as the resulting
manifold is no longer a smooth manifold with corners.  We adopt a
different strategy here: we \emph{change the smooth structure} on
$M$ to obtain a new smooth manifold with boundary $\Mlog$ before the
blow-up.  This process removes the ambiguity surrounding what sort of
object the blown-up manifold becomes, but at the cost of introducing
logarithmic singularities in the metric coefficients.  All methods
require fixing a product structure in $X=\pd M$ near $S_{\pm}$, but the
results do not depend on the choice of product structure.  We
emphasize that this change of smooth structure will be employed only
in the final stage of our arguments (denoted ``Full Asymptotics'' in
the sketch from the introduction), when we perform the radiation field
blowup and deduce our asymptotic expansion at $\scri^+$; in the
intervening stage, at which we iteratively invert the reduced normal
operator of $L$ globally on $\pa M,$ we are using the original smooth structure.

In what follows, the coordinates on the new, ``logified'' space are
denoted by the same letters but in different fonts.  We typically
distinguish the logified function spaces with a subscript
``$\log$''.

Our assumptions on the metric $g$ imply that $dv$ is non-degenerate
in a neighborhood of $S_{\pm}$.  In particular, we now consider an
atlas of coordinate charts $\varphi_{\alpha}: U_{\alpha}\to
V_{\alpha}\subset \reals^{n}_{+}$ of this neighborhood so that $\rho$
and $v$ are always two of the coordinates.  (We denote the remaining
coordinates by $\varphi^{y}_{\alpha}$.)  Note that restricting our
attention to such charts fixes a product decomposition near
$S_{\pm}$.  

Because $S_{\pm}$ are compact, there is some constant $C$ so that $\{
\rho = 0, |v| \leq C\}$ is covered by the union of the $U_{\alpha}$.
Fix now a function $\chi \in C^{\infty}_{c}(\reals)$ so that $\chi
\equiv 1$ near $0$ and $\chi (v) \equiv 0$ for $|v|\geq C$.  

We now introduce the functions $\rholog = \rho$ and $\vlog = v + \chi
(v) \mass \rho \log \rho$ and observe that the restriction of $\vlog$
to $X$ agrees with $v$.  We change the smooth structure of the
manifold by defining a new atlas in the neighborhoods $U_{\alpha}$.
Indeed, we define charts $\tilde{\varphi}_{\alpha}: U_{\alpha}\to
\tilde{V}_{\alpha}\subset \reals^{n}_{+}$ on $\Mlog$ by
\begin{equation*}
  \tilde{\varphi}_{\alpha} = (\rholog, \vlog, \varphi^{y}_{\alpha}).
\end{equation*}
In other words, we change the smooth structure of $M$ by asking that
the function $\vlog$ (rather than $v$) be smooth.  We also use the
notation $\ylog = y$ in these coordinates.

Because $v = \vlog - \chi \mass \rholog \log \rholog$, smooth functions on
$M$ (i.e., those admitting expansions in $(\rho, v, y)$ with
nonnegative integer exponents) no longer are smooth on $\Mlog$ but
instead admit expansions in $\rholog$, $\rholog\log \rholog$, $\vlog$, and $\ylog$
with nonnegative integer exponents.  

We now define the algebras of functions and of differential operators
with mildly singular coefficients that we employ.

\begin{definition}
  \label{def:CIlog}
  We let $\CI (\Mlog)$ denote the coefficient ring of functions smooth
  on $\Mlog$ (i.e., $M$ equipped with the new smooth structure), while
  we let $\CIlog(\Mlog)$ denote the coefficient ring consisting of smooth
  functions of $\rholog$, $\rholog \log \rholog$, $\vlog$, and $\ylog$.
\end{definition}

Observe that $\CIlog$ is also the set of distributions conormal to $X
= \pa \Mlog$ that are polyhomogeneous with index set
\begin{equation*}
  \logsmoothset = \{ (k, j) : k = 0, 1, 2, \dots, j = 0, 1, \dots, k\}.
\end{equation*}

To clarify which manifold we are working on, we introduce the notation
$$
\taut: M \to \Mlog
$$
for the tautological map between these two manifolds.  We let
$$\taut^*: \CI(\Mlog) \to \CIlog(M)$$ denote the natural pullback map,
and also, by modest abuse of notation, let
$$\taut_*: \CI(M) \to \CIlog(\Mlog),$$ denote pullback under
$\taut^{-1}.$  We will employ the analogous notation for pushforward
(and pullback!) of vector fields as well.

\begin{definition}
  We say that $P \in \Diffblog^{k}(\Mlog)$ if $P \in \CIlog(\Mlog)
  \otimes \Diffb^{k}(\Mlog)$.   In other words, $P \in
  \Diffblog^{k}(\Mlog)$ if there are coefficients $a_{\alpha}\in
  \CIlog(\Mlog)$ so that
  \begin{equation*}
    P = \sum_{\smallabs{\alpha} \leq k} a_{\alpha} D^{\alpha},
  \end{equation*}
  where $D^{\alpha}$ are monomials in the vector fields $\rholog
  D_{\rholog}$, $D_{\vlog}$, and $D_{\ylog_{j}}$.
\end{definition}

Even though the lift of $\taut_* L$ of $L$ to $\Mlog$ lives in this
space, we include here a more general class of operators for future
work.  In particular, we also allow terms of the form
$\rholog \log \rholog D_{\rholog}$, which are not in $\Diffblog$.  We
consider the slightly larger space
$\Diffblogbad^{k}(\Mlog)$.  Elements of this space have the form
\begin{equation*}
  \sum_{\smallabs{\alpha}\leq k}a_{\alpha}D^{\alpha},
\end{equation*}
where $a_{\alpha}\in\CIlog$ and $D^{\alpha}$ are monomials in the
vector fields $\rholog D_{\rholog}$, $\rholog \log \rholog
D_{\rholog}$, $D_{\vlog}$, and $D_{\ylog_{j}}$.  Observe that
$\Diffblog^{k} \subset \Diffblogbad^{k}$.\footnote{Unfortunately, this larger
space is not a graded algebra (and $\Diffblogbad^{1}/\Diffblogbad^{0}$
is not a Lie algebra), but we avoid these problems by working with
$\Diffblog$ when possible.}

Let $\cI \subset \CI$ denote the ideal of smooth functions
vanishing at $S_{+}$, while $\cIlog\subset \CIlog$ is the ideal of
$\CIlog$ functions vanishing at $S_{+}$.  In other words, $f \in \cI$
if there are smooth functions $a_{1}$ and $a_{2}$ so that $f = \rholog
a_{1} + \vlog a_{2}$, while $f \in \cIlog$ if there are $\CIlog$
functions $a_{1}$, $a_{2}$, and $a_{3}$ so that $f = \rholog a_{1} +
\vlog a_{2} + (\rholog \log \rholog )a_{3}$.

We now define the module $\cMdifflog \subset \Diffblog^{1}$ to be the
module of vector fields logarithmically tangent to $\rho = 0$ and
$S_{+},$ i.e.\ that map $\rholog,\ \vlog$ to $O(\rholog) + O( \rholog \log
\rholog)+O(\vlog).$ Over $\CIlog,$ this module is generated by $\rholog D_{\rholog}$,
$\rholog D_{\vlog}$, $\vlog D_{\vlog}$, and $D_{\ylog}$.

Finally, we define the ``bad module'' $\badmod\subset \Diffblogbad^{1}$
as the corresponding module in the larger space, so it is generated
over $\CIlog$ by $\rholog D_{\rholog}$, $\rholog \log \rholog
D_{\rholog}$, $\rholog D_{\vlog}$, $\rholog \log \rholog D_{\vlog}$,
$\vlog D_{\vlog}$, and $D_{\ylog _{j}}$.  

Observe that just as the module $\cM$ maps $\cI$ to itself, 
$\cMdifflog$ preserves $\cIlog$.  The ``bad module'' $\badmod$ maps
$\cI$ to $\cIlog$.

\begin{lemma}
  \label{lemma:changeofmodule}
  We may characterize $\cMdifflog$ as
  \begin{equation*}
    \cMdifflog = \CIlog \otimes \cMdiff (\Mlog)
  \end{equation*}
  and the following inclusion holds for $\badmod$:
  \begin{equation*}
    \badmod \subset \cMdifflog + (\log \rholog) \cMdifflog.
  \end{equation*}

  Moreover, $$\taut_* \cMdiff\subset \badmod,$$ while $$\badmod\subset
  \taut_* \cMdiff(M) + \taut_* (\log \rho)
  \cMdiff(M).$$
\end{lemma}

\begin{proof}
  The first statement follows because $\cMdiff$ and $\cMdifflog$ are
  generated by the same vector fields but over different rings.  The
  second statement follows by examination of the generators of
  $\badmod$.  

  For the final statement, we need only calculate the lifts of the
  generators.  For instance,
  \begin{equation*}
   \taut_* \rho D_{\rho} = \rholog D_{\rholog} + \rholog \left(
      \frac{\pd \vlog}{\pd \rho}D_{\vlog}\right) = \rholog D_{\rholog}
    + \rholog m (\log \rholog + 1)D_{\vlog} + O(\rholog^{2}\log \rholog)D_{\vlog}.
  \end{equation*}
\end{proof}

A consequence of Lemma~\ref{lemma:changeofmodule} is that the passage
from $M$ to $\Mlog$ does not materially change the b-Sobolev spaces.
In particular, we have the following equalities for all $s$ and
$\gamma$:
\begin{equation}\label{sobolevlog}
 \iota^* \Hb^{s,\gamma-0}(\Mlog) = \Hb^{s,\gamma-0}(M).
\end{equation}
This means that, provided we are willing to lose a small amount of
regularity and decay, neither the Sobolev nor conormal spaces change.
A further consequence of this fact is that
Proposition~\ref{prop:radial-b-estimate} implies that module
regularity under $\cM(M)$ immediately implies module regularity under
$\cMdifflog (\Mlog)$ and $\badmod$.
In particular,
the bad module loses just an epsilon relative to the good
one owing to log terms: 

\begin{proposition}
  \label{prop:bad-module}
  For each $k\in \NN$, we
  have 
  \begin{equation}
    \label{badmodsobolev} \cMdiff^N u \in
    \Hb^{s,\gamma} (M)\ \forall N \Longrightarrow \cMdiff^N \badmod^k
    \taut_* u \in \Hb^{s,\gamma-0} (\Mlog)\ \forall N.
  \end{equation}
\end{proposition}

We now easily verify the following:
\begin{proposition}
  The space $\Diffblog^{1}(\Mlog) / \Diffblog^{0}(\Mlog)$ consisting
  of b-vector fields with coefficients in $\CIlog$ is a Lie algebra;
  $\Diffblog^{*}(\Mlog)$ is a graded algebra.  
\end{proposition}

\begin{proof}
  The only new ingredient compared to the usual, smooth, case is the
  fact that
  \begin{equation*}
    [\rholog, D_{\rholog}, \rholog\log \rholog] = \imath^{-1}(\rholog
    \log \rholog + \rholog) \in \CIlog.
  \end{equation*}
\end{proof}

An essential ingredient in our iterative argument will be the
following refinement of Lemma~\ref{lemma:LmodM}; this is essentially
the main point of our change of variables from $v$ to $\vlog$, which
makes the $-4\mass \rho D_{v}^{2}$ term in the operator disappear.
\begin{lemma}
  \label{lemma:Lnormop}
We have
  \begin{equation}
    \label{eq:modlogs}
   \taut_* L = 4\pd[\vlog]\left( \rholog\pd[\rholog] + \vlog
      \pd[\vlog]\right) + \badmod^{2}
  \end{equation}
\end{lemma}

\begin{proof}
  We note that in the coordinate change from $v$ to $\vlog$, we have
  \begin{equation*}
  \taut_*  \pd[v] = (1 + \chi' (v) \mass \rholog \log \rholog) \pd[\vlog], \
    \taut_* \pd[\rho] =\pd[\rholog] +  \chi (v) \mass (1 + \log
    \rholog) \pd[\vlog]
  \end{equation*}
  and hence
  \begin{align*}
  \taut_*  v\pd[v] &= \left( \vlog - \chi \mass \rholog \log
      \rholog\right)\left( 1 + \chi' \mass \rholog
      \log\rholog\right)\pd[\vlog] \\
   \taut_* \rho \pd[\rho] &= \rholog\pd[\rholog] + \chi \mass \rholog
    (1 + \log \rholog) \pd[\vlog]
  \end{align*}
  Applying Lemma~\ref{lemma:changeofmodule} yields~\eqref{eq:modlogs}.
\end{proof}

The following lemma is useful in Section~\ref{sec:asympt-expans}; it
shows that additional vanishing at $S_{+}$ in fact improves regularity.
\begin{lemma}
  \label{lemma:inclusion}
  $\displaystyle \cI\subset \Psib^{-1}\cMdiff$ and $\cIlog\subset \Psib^{-1}\cMdifflog$.
\end{lemma}

\begin{proof}
  We prove the lemma in the first case; the proof is nearly identical
  in the logified setting.

  It suffices to show that $\rho, v \in
  \Psib^{-1}\cMdiff$.  To do this, note that as a composition of
  operators,
  \begin{equation*}
    \begin{pmatrix}
      \rho \pd[\rho] \\ \pd[v] \\ \pd[y]
    \end{pmatrix}
    \circ v \in
    \begin{pmatrix}
      \cMdiff \\ \vdots \\ \cMdiff
    \end{pmatrix}.
  \end{equation*}
  The vector-valued b-operator on the left has a left-invertible
  symbol and thus has a left inverse in $(\Psib^{-1}, \dots ,
  \Psib^{-1})$ modulo $\Psib^{-\infty}$, hence (since certainly
  $\Psib^{-\infty}\subset \Psib^{-1}\cMdiff$) we have
  \begin{equation*}
    v \in \Psib^{-1}\cMdiff.
  \end{equation*}
  The proof for $\rho$ proceeds in the same way.
\end{proof}

We now discuss how asymptotic expansions are transformed by the
logification process.
\begin{proposition}\label{proposition:asymptotic.log}
Let $u \in \phg{E}(M).$
For each $j \in \NN$ let 
$$
E(j) \equiv \{(z-j\imath, \ell): (z,k) \in E, 0 \leq \ell \leq k + j\}.
$$
Let
$$
E' \equiv E(0) \cup E(1)\cup E(2) \cup \dots
$$
Then $$\taut_* u \in \phg{E'}(\Mlog).$$
\end{proposition}
We note that an alternative definition of $E',$ since $E$ is an index
set and therefore closed under $(z,k) \to (z-\imath,k),$ is in terms
of extended unions as the set
$$
E' = E \eu E_1 \eu E_2 \eu \dots
$$
with $E_j=\{(z-j\imath,k): (z,k) \in E\}.$

\begin{proof}
We employ the method of testing by radial vector fields.  Note that
$$
\Rlog \equiv \rholog D_\rholog=\taut_* \left( \rho D_\rho-\chi m 
 \rho (1+\log \rho)  D_v\right) + \chi ' m \rholog^{2}\log \rholog (1
+ \log \rholog )D_{\vlog}.
$$
Thus, 
$$
\taut^* (\Rlog-z)^{k+1} = (\rho D_\rho-z)^{k+1} + F,
$$
where $$F \in \rho \Diffblog^{k+1}(M)+\rho \log \rho
\Diffblog^{k+1}(M),$$ and more precisely, $F$ is a sum of products of
smooth b-vector fields times coefficients containing powers of $\rho$
and $\rho \log
\rho$ between $1$ and $k+2$ (though we do not need this
characterization).

Now for any index set $G$ let $S$ denote the shift operation with
increase of multiplicity:
$$
S(G) \equiv \{(z-\imath, k+1): (z,k) \in G\}.
$$
Hence $E(j+1) = S(E(j))$ and $E'$ is closed under $S.$

Now, since application of b-vector fields preserves index sets while
multiplication by $\rho$ and $\rho \log \rho$ shifts them according to
the map $S$,
we find in general that if $w$ has index set $G$ on $M$ and if $(z,k) \in G,$ then
$$
\taut^* (\Rlog-z)^{k+1}w \in \phg{G_z}(M),
$$
where
$$
G_z = (G \setminus (z,k) ) \cup S(G) \cup S^{2}(G) \cup \dots.
$$
Letting $z$ be the value of in $G$ with
largest imaginary part, we then see that this process yields an index
set with strictly smaller imaginary parts.  (If there are several with
same imaginary part, we must of course repeat the process finitely
many times.)

Now we apply this argument iteratively to $u$: if $u$ has index set
$E,$ i.e., improved decay under application of 
$$
\prod_{(z,k) \in E} (\rho D_\rho-z)^{k+1},
$$
then
we pick $(z_0,k_0)$ with largest imaginary part (again iterating if this is
not unique) and note that
$$
\taut^* (\Rlog-z_0)^{k_0+1}u \in \phg{E_z}(M),
$$
where now $E_z$ is an index set with smaller imaginary part, and is
contained in $E'.$
Continuing inductively (and remembering at every stage that $E'$ is
conserved by $S$) we see that $u$ has improved decay under application of
$$
\taut^* \prod_{(z,k) \in E'} (\Rlog-z)^{k+1}.
$$
Pushing forward (and recalling that the scale of weighted Sobolev spaces is
essentially unchanged by $\taut_*$) we see that $\taut^* u$ has
improved decay under
$$
\prod_{(z,k) \in E'} (\Rlog-z)^{k+1},
$$
as desired.\end{proof}

We will in practice need the version of this result that deals with
the rougher expansions with coefficients conormal at $S_+.$ To this
end, we say that $u$ lies in the $L^{2}$-based conormal space
$I^{(s)}(\Lambda^{+})$ if $u \in H^{s}(X)$ and $A_{1}\dots A_{k}u \in
H^{s}(X)$ for all $k \in \NN$ and $A_{j} \in \cMdiff$.  We then have
the following
\begin{proposition}\label{proposition:asymptotic.log2}
If a distribution $u$ on $M$ conormal with respect to $N^* S_+$ enjoys
an expansion
\begin{equation}
u \sim  \sum_E a_j(v,y) \rho^{\imath \sigma_j} (\log \rho)^{k_j}
\end{equation}
with index set $E$ and with
$$
a_j \in I^{(q_0-\Re (\imath\sigma_j))}(\Lambda^{+})
$$
then on $\Mlog,$ $\taut_* u$ has an expansion
$$
\taut_* u \sim  \sum_{E'} b_j(v,y) \rho^{\imath \sigma_j} (\log \rho)^{k_j}
$$
where
$$
b_j \in I^{(q_0-\Re (\imath\sigma_j)-0)} (\Lambda^{+}).
$$
and where the index set
$$
E' \equiv E(0) \cup E(1)\cup E(2) \cup \dots
$$
with
$$
E(j) \equiv \{(z-j\imath, k+\ell): (z,k) \in E, 0 \leq \ell \leq j\},
$$
\end{proposition}
\begin{proof}
The proof is just as in Proposition~\ref{proposition:asymptotic.log},
using the oscillatory testing characterization (Proposition~\ref{proposition:asymptotic.conormal})
but with the additional feature that we note that the logarithmic
change of variables shifts conormal orders by $\ep$ for any
$\ep>0.$
\end{proof}

\section{The radiation field blow-up}
\label{sec:radiation-field-blow}

In this section we recall from~\cite{radfielddecay} the construction of the
manifold $[\Mlog ; S]$ on which the radiation field lives.

We now \emph{blow up} $S= \{ \vlog = \rholog = 0\}$ in $\Mlog$ by
replacing it with its inward pointing spherical normal
bundle.\footnote{The reader may wish to consult~\cite{Melrose:APS} for
  more details on the blow-up construction than we give here.}  This
process replaces $\Mlog$ with a new manifold with corners $[\Mlog; S]$
on which polar coordinates in $\rholog, \vlog$ are smooth, and depends
only on $S$ and the smooth structure of $\Mlog$.  The blow-up comes
equipped with a natural blow-down map $[\Mlog ; S] \to \Mlog$ which is
a diffeomorphism on the interior.  $[\Mlog ; S]$ is a manifold with
corners with several boundary hypersurfaces: the closure of the lifts
of $C_{0}$ and $C_{\pm}$ to $[\Mlog ; S]$, which we still denote
$C_{0}$ and $C_{\pm}$, and $\scri$, which we define as the lift of $S$
to $[\Mlog ; S]$.  Further, the fibers of $\scri$ over the base, $S$,
are diffeomorphic to intervals, and indeed, the interior of a fiber is
naturally an affine space (i.e., these interiors have $\RR$ acting by
translations, but there is no natural origin).

Given $\vlog$ and $\rholog$, the fibers of the interior of $\scri$ in
$[\Mlog ; S]$ can be identified with $\RR$ via the coordinate $s =
\vlog / \rholog$.  In particular, $\pd[s]$ is a well-defined vector
field on the fibers.

In what follows, we note that $s=\vlog/\rholog$ is a smooth
coordinate along $\scri^+,$ and $s^{-1},\ \rholog$ are respectively the
defining functions of (the lift of) $C_+$ and $\scri^+.$  If we are
interested in studying forward solutions, this corner and the two
faces meeting at it are the only places where $u$ has nontrivial asymptotics.

With the notation of the previous sections in hand, we finally restate
our main theorem in more detail:
\begin{maintheorem}
Let $(M,g)$ be a non-trapping Lorentzian scattering manifold, and let 
$$
\Box_g u=f
$$
with $u \in \CmI(M),$ $f \in \dCI(M).$  Assume that $u$ is a forward solution.
Then $u$ lifts to $[\Mlog; S]$ to have a joint polyhomogeneous expansion at
all boundary faces, vanishing except at the face $C_+$ and the front
face $\scri^+$ of the blowup of $S_+.$  At that pair of faces the
powers in the polyhomogeneous expansion are given by $\Etot$ described
above in Section~\ref{subsec:b-geometry}, hence with terms that are powers of a defining
function at $C_+$ described
in terms of poles of the family of $P_\sigma$ and at $\scri^+$
given by terms $\rho^j (\log\rho)^{\ell}$ for $\ell=0,\dots, 2j$ for $m\neq
0$ and simply by $\rho^j$ for $m=0.$
\end{maintheorem}

\section{Asymptotic expansions}
\label{sec:asympt-expans}

We are now ready to derive the asymptotic expansion of solutions to
the wave equation on $\Mlog,$ thereby proving Theorem~\ref{mainthm}.
In the case $m=0,$ such an asymptotic expansion
was derived in~\cite{radfielddecay}, but adapting the argument given
there would be rather cumbersome.  Instead, we proceed with a
different, and (we hope) more transparent, argument which in fact
yields more.

\emph{The proof will proceed in two steps, one for each boundary face
  of the radiation field blowup.}  By
Proposition~\ref{proposition:doublemellin}, it will suffice to obtain
the asymptotics at each boundary face with uniform control of error
terms at the other face.  To begin, we work at $C_+;$ while this
argument will initially appear to be a global one near $\pa M,$ the
worsening error terms at $S_+$ will mean that this first step will
yield only the asymptotics at $C_+,$ uniformly up to $\scri^+$ after the
radiation field blowup.

As we use the following spaces many times, it is convenient to
introduce a compact notation:
\begin{definition}
  For $\varsigma, s \in \RR$, we let
  \begin{equation*}
    \B (\varsigma, s) = \hol ( \CC_{\varsigma}) \cap
    \smallang{\sigma}^{-\infty}L^{\infty}L^{2}(\RR ; I^{(s)}(\Lambda^{+})).
  \end{equation*}
  Here $\hol (\CC_{\varsigma})$ is the space of holomorphic functions
  on the half-space $\CC_{\varsigma}$ defined in
  Definition~\ref{def:holo}.  
\end{definition}

\subsection{Asymptotics at $C_+$}
We start by recalling a portion of the argument of~\cite{radfielddecay}
yielding asymptotic expansions at $C_{+}$.  

As in Lemma~\ref{lemma:Lnormop}, we write the operator $L = N(L) + E$,
where $E \in \rholog \Diffb^{2}(M)$.  We let $R_{\sigma}$ be
the family of operators intertwining $E$ with the Mellin transform,
i.e., satisfying  
\begin{equation*}
  \Mellin \circ E = R_{\sigma} \circ \Mellin.
\end{equation*}
$R_{\sigma}$ is thus an operator on meromorphic families in $\sigma$
in which $\rholog D_{\rholog}$ is replaced by $\sigma$ and
multiplication by $\rholog$ translates the imaginary part.

Note that since the mass term only appears with an $O(\rho)$ relative
to the main terms in $L$ when written as a b-operator, $N(L)$ is
independent of $m,$ hence agrees with the expression found in
\cite{radfielddecay}.

By Lemma~\ref{lemma:Lnormop} we have the following result on the
mapping properties of $R_{\sigma}$.  The mapping properties of
$R_{\sigma}$ are slightly worse here than in our previous work owing
to the presence of a term of the form $\rho D_{v}^{2}$ in $L$ in the
long-range setting.

\begin{lemma}[\cite{radfielddecay}, Lemma 9.1]
  \label{lemma:remainder}
  For each $\nu, k, \ell, s$, the operator family $R_{\sigma}$ enjoys
  the following mapping properties:
  \begin{enumerate}
  \item $R_{\sigma}$ enlarges the region of holomorphy at the cost of
    regularity at $\Lambda^{+}$:
    \begin{align}
      \label{eq:remainder1}
      R_{\sigma} : &\hol(\CC_{\nu}) \cap
      \smallang{\sigma}^{-k}L^{\infty}L^{2}(\RR ;
      I^{(s)}(\Lambda^{+}))  \notag \\
      \to &\hol (\CC_{\nu + 1}) \cap \smallang{\sigma}^{-k+2}
      L^{\infty}L^{2}(\RR ; I^{(s-2)}(\Lambda^{+})) 
    \end{align}
 \item If $f_{\sigma}$ vanishes near $\overline{C_{-}}$ for $\Im
   \sigma \geq - \nu$, then $R_{\sigma} f_{\sigma}$ also vanishes
   near $\overline{C_{-}}$ for $\Im \sigma \geq -\nu - 1$.
  \end{enumerate}
\end{lemma}

As discussed above, we transform the equation
$$
\Box_g u=f
$$
by rescaling and conjugation to rewrite it as
$$
L w=g
$$
where
\begin{equation}
L\equiv \rho^{-(n-2)/2-2}\Box_g \rho^{(n-2)/2},
\end{equation}
$$
w=\rho^{-(n-2)/2} u\in \CmI(M) ,\quad g=\rho^{-(n-2)/2-2} f \in \dCI(M).
$$

Thus, suppose $Lw = g$, where $g \in \dCI (M)$ and $u$ vanishes
in a neighborhood of $\overline{C_{-}}$.  Taking the Mellin transform,
we obtain
\begin{equation}
  \label{eq:toiterate}
  P_{\sigma} \tilde{w}_{\sigma} = \tilde{g}_{\sigma} - R_{\sigma} \tilde{w}_{\sigma}.
\end{equation}
As $g\in \dCI(M)$, we have
\begin{equation*}
  \tilde{g}_{\sigma} \in \B (C, s') \text{ for all }C, s'.
\end{equation*}
Because $\rho^{(n-2)/2}w$ lies in some $\Hb^{s,\gamma}(M)$, we
have
\begin{equation}
  \label{eq:utildesobolev}
  \tilde{w}_{\sigma} \in \hol (\CC_{\varsigma_{0}}) \cap
  \smallang{\sigma}^{\max(0,-s)}L^{\infty}L^{2}(\RR ; H^{s}),
\end{equation}
where $\varsigma_{0} = \gamma - (n-2)/2$.  By reducing $s$, we may
assume that $s+ \gamma < 1/2$ so as to be able to apply the module
regularity results of Proposition~\ref{prop:radial-b-estimate}.  We
may also arrange that $\tilde{w}_{\sigma}$ vanishes in a neighborhood
of $\overline{C_{-}}$ in $X$ because, by hypothesis, $w$ vanishes near
$\overline{C_{-}}$ in $M$.

Because the metric is non-trapping, we know that $w$ has module
regularity with respect to $\cM$, and so by~\cite[Lemma 2.3]{radfielddecay},
\begin{equation*}
  \tilde{w}_{\sigma} \in \B(\varsigma_{0}, - \infty),
\end{equation*}
and thus, by interpolation with~\eqref{eq:utildesobolev},
\begin{equation*}
  \tilde{w}_{\sigma} \in \B (\varsigma_{0}, s-0).
\end{equation*}
In particular, $R_{\sigma}\tilde{w}_{\sigma}$ (and hence
$P_{\sigma}\tilde{w}_{\sigma}$) lies in
\begin{equation*}
  \B (\varsigma_{0} + 1, s-2 - 0)
\end{equation*}

Because $P_{\sigma}\tilde{w}_{\sigma}$ is known to be holomorphic in a
larger half-plane, we can now invert $P_{\sigma}$ to obtain meromorphy
of $\tilde{w}_{\sigma}$ on this larger space: by
Propositions~\ref{prop:fredholm} and \ref{prop:nontrappingestimates},
$P_{\sigma}$ is Fredholm as a map
\begin{equation*}
  \cX^{s_{\tow}}\to \cY^{s_{\tow}-1}, 
\end{equation*}
and $P_{\sigma}^{-1}$ has finitely many poles in any horizontal strip
$\Im \sigma \in [a,b]$.  Moreover, $P_{\sigma}^{-1}$ satisfies
polynomial growth estimates as $\smallabs{\Re \sigma}\to \infty$.
Here we recall from Section~\ref{sec:Fredholm}
that given any $\varsigma '$, in order for $P_{\sigma}$ to be Fredholm
for $\sigma \in \CC_{\varsigma'}$, the (constant) value $s(S_{+})$
assumed by the variable Sobolev order $s_{\tow}$ near $S_{+}$ must
satisfy $s(S_{+}) < 1/2 - \varsigma'$; thus as one enlarges the domain
of meromorphy for $\tilde{w}_{\sigma}$, one needs to relax the control
of the derivatives.  Thus $\tilde{w}_{\sigma}$ is obtained by applying
$P_{\sigma}^{-1}$ to the right hand side
of~\eqref{eq:remainder1}; this term is
meromorphic in $\CC_{\varsigma_{0}+1}$ with values in 
\begin{equation}
  \label{eq:rhs1}
  \smallang{\sigma}^{-\infty}L^{\infty}L^{2}(\RR ; H^{\min (s-1-0, 1/2-\varsigma_{0}-1-0)})
\end{equation}
with (finitely many) poles in this strip, arising from the poles of
$P_{\sigma}^{-1}$.  Here (and below) we
are ignoring the distinction between $\cX^{s_{\tow}}$ and $H^{s}$ as
$\tilde{w}_{\sigma}$ is trivial by hypothesis on the set where the
regularity in the variable-order Sobolev space differs from $H^{s}$.

Now we can improve our description of the remainder terms (going back
to Lemma~\ref{lemma:remainder} for the description of
$R_\sigma \tilde{w}_\sigma$) since $P_{\sigma}$ maps the expression
in question to
$\smallang{\sigma}^{-\infty}L^{\infty}L^{2}(\RR ; I^{(s-2-0)}(\Lambda^{+}))$.  Thus
the term~\eqref{eq:rhs1} must in fact be
meromorphic with values in the \emph{conormal} space
\begin{align*}
  &\smallang{\sigma}^{-\infty}L^{\infty}L^{2}(\RR ; I^{\min (s-1-0,
    1/2-\varsigma_{0}-1-0)}(\Lambda^{+})),
\end{align*}
by propagation of singularities away from radial points
(Proposition~4.1 of~\cite{radfielddecay}) and the first case of Theorem~6.3
of~\cite{Haber-Vasy:Radial}, which deals with propagation of
Lagrangian regularity into conic Lagrangian submanifolds of radial points.\footnote{Here Theorem~6.3 is applied
  pointwise in $\sigma$; the result there is not stated in terms of
  bounds (just as a membership in the claimed set), but just as in the
  case of Proposition~\ref{prop:radial-b-estimate} here, estimates can
  be recovered from the statement of Theorem~6.3 by the closed graph
  theorem or alternatively recovered from examination of the proof,
  which proceeds via such estimates.}

Thus we have now shown that
\begin{align}
  \label{eq:tu-step-1}
  \tilde{w}_{\sigma} \in &\B(\varsigma_{0}+1, \min (s-1-0,
  1/2-\varsigma_{0}-1-0)) \\
  &+ \sum_{\substack{(\sigma_j, m_j) \in \ressetinit \\ \Im \sigma_{j}
    > -\varsigma_{0}-1}}(\sigma - \sigma_{j})^{-m_{j}} a_{j},
\end{align}
where
\begin{equation*}
  a_{j} \in \B (\varsigma_{0}+1, \Im\sigma _{j} + 1/2 - 0).
\end{equation*}
Here the conormal regularity of the coefficients of the polar part
follows from the Cauchy integral formula.

We now iterate this argument as follows.  (The argument
is simpler than the analogous argument in \cite{radfielddecay}, as we will
allow derivative losses in our conormal spaces that we will recoup
later.)

Assume inductively that
\begin{align}
  \label{eq:indhypothesis}
  \tilde{w}_{\sigma} &\in \B (\varsigma_{0} +N, \min (s-N-0,
  1/2-\varsigma_{0}-N-0)) + \dots  \\
  &\quad  + \sum_{\substack{(\sigma_j, m_j) \in \mathcal{E}_0\eu \dots
      \eu \mathcal{E}_{N},\\ \Im\sigma_{j} > -\varsigma_{0} - N}}(\sigma -
  \sigma_{j})^{-m_{j}}a_{j}, \notag
\end{align}
with
\begin{equation*}
  a_{j}\in \B (\varsigma_{0} + 2N, 1/2+ \Im \sigma_{j}-0).
\end{equation*}
By Lemma~\ref{lemma:remainder},
\begin{align}
  \label{eq:Rsigmaiterate}
  R_{\sigma}\tilde{w}_{\sigma} &\in \B(\varsigma_{0} + N + 1, \min
  (s-N-2-0, 1/2-\varsigma_{0}-N-1-0)) + \dots \\
  &\quad + \sum_{\substack{(\sigma_j, m_j)\in \mathcal{E}_0\eu \dots
      \eu \mathcal{E}_{N}\\ \Im\sigma_{j}> -\varsigma_{0}-N}} (\sigma -
  (\sigma_{j}-\imath))^{-m_{j,1}}a_{j}'\notag\\
\end{align}
where
\begin{align*}
  &a_{j}'\in \B(\varsigma_{0} + N, 1/2 + \Im (\sigma_{j} - \imath)-0),
\end{align*}
We remark that all the shifted poles in the above expressions lie in
the new index set
$$
\mathcal{E}_1\eu \dots \eu \mathcal{E}_{N+1}.
$$
Now we may apply $P_{\sigma}^{-1}$ as above to
solve for $\tilde{w}_{\sigma}$ on the left hand side
of~\eqref{eq:toiterate} and find that~\eqref{eq:indhypothesis} holds
for all $N,$ since the new poles introduced by the operator family are
given by the extended union with $\mathcal{E}_0$.  Inverse Mellin transforming this result
then yields the following asymptotic expansion.
We thus have the following:
\begin{proposition}
  \label{proposition:asympeq}
  Let $\masslessres$ be the massless resonance index
  set (with $\varsigma_{0}$ chosen to ignore those resonances where
  $\tilde{w}_{\sigma}$ is a priori holomorphic).  Then on $M,$
  \begin{equation*}
    w = \sum_{\substack{(\sigma_j,k)\in \masslessres\\  \Im\sigma_j >
      -l}}\rholog^{\imath\sigma_j}(\log\rholog)^{k}a_{jk} + w',
  \end{equation*}
  where, for $C = s + \varsigma_{0}$, 
  \begin{equation*}
    w' \in \rholog^{l}\Hb^{\min(C-l-0, 1/2-\varsigma_{0}-l-0),\gamma}(M).
  \end{equation*}
 The coefficients
  $a_{jk}$ are $\CI$ functions of $\rholog$ taking values in
  $I^{(1/2-\Re(\imath\sigma_{j})-0)}$ and are supported in
  $\overline{C_{+}}$.

  Moreover on the logified manifold $\Mlog$ we have
  \begin{equation*}
    \taut_* w = \sum_{\substack{(\sigma_j,k)\in \resset\\  \Im\sigma_j >
        -l}}\rholog^{\imath\sigma_j}(\log\rholog)^{k}b_{jk} + w',
  \end{equation*}
  where the coefficients $b_{jk}$ have the same properties as the
  $a_{jk}$ and
  \begin{equation*}
    w' \in \rholog^{l}\Hb^{\min(C-l-0, 1/2-\varsigma_{0}-l-0),\gamma}(\Mlog).
  \end{equation*}
\end{proposition}

\begin{proof}
The proposition follows by taking the inverse Mellin transform of
  $P_{\sigma}^{-1}$ applied to~\eqref{eq:Rsigmaiterate}: the polar
  terms yield the terms in the sum, while the ``remainder'' term
  arises from the first $2N$ terms in~\eqref{eq:Rsigmaiterate}.  This
  yields the expansion with index set $\masslessres$ on $M.$ Now
  applying Proposition~\ref{proposition:asymptotic.log2} gives us the
  corresponding expansion with index set $\resset$ on $\Mlog.$
\end{proof}

\begin{remark}
Note that the
expansion appears somewhat unsatisfactory as the conormal
regularity declines as the power of $\rholog$ increases, but we will
cope with this inconvenience later.
\end{remark}

\begin{remark}
  We remark that we can recover a form of Price's law in the setting
  of very short-range perturbations of Minkowski space.  If $(M,g)$
  decays sufficiently quickly to Minkowski space, then the induced
  operator $P_{\sigma}$ on the boundary agrees with the operator in
  Minkowski space.  The poles of $P_{\sigma}^{-1}$ can be computed
  explicitly and lie at $-\imath \frac{n-1}{2} - \imath j$, where
  $j\in \mathbb{N}$; when the spacetime dimension is odd the resonant
  states corresponding to these poles are supported on $S_{+}$.  After
  applying the non-local operator $P_{\sigma}^{-1}$ to the residue
  from one of these poles, the pole shifts down and the new residue
  (denoted $a_{jk}$ in Proposition 9.3) is supported in
  $\overline{C_{+}}$.  The consequences of this spreading of support
  are discussed below in Remark~\ref{rem:price}.
\end{remark}

We now continue our discussion of asymptotic expansions working
exclusively on $\Mlog.$  In what follows, though the exact value of
the constant $C$ is irrelevant, it may be taken to be $s + \varsigma_{0}$.

As a consequence of Proposition~\ref{proposition:asympeq}, we have
\begin{equation*}
  w' = \left( \prod_{(\sigma_{j}, k)\in \resset(\varsigma_{0}),
      \Im\sigma_{j}> -l} (\rholog D_{\rholog} - \sigma_{j})\right) w \in \rholog^{l}\Hb^{\min (C-l-0,
    1/2-\varsigma_{0}-l-0), \gamma}(M).
\end{equation*}
Now by Proposition~\ref{prop:radial-b-estimate}, $w$ enjoys module
regularity with respect to $\rholog^{l}\Hb^{s',\gamma}(\Mlog)$ for
some $s'$.  Thus for all $N'$,
\begin{equation*}
  \cMlog^{N'}\left( \prod_{(\sigma_{j}, k)\in \resset(\varsigma_{0}),
      \Im\sigma_{j}> -l} (\rholog D_{\rholog} - \sigma_{j})\right) w
  \in \rholog^{l}\Hb^{s',\gamma}(\Mlog);
\end{equation*}
here we have of course used the fact that all the factors $(\rholog
D_{\rholog} - \sigma_{j})$ lie in $\cMlog$.  Interpolation now yields for
all $N$
\begin{equation*}
  \cMlog^{N}\left( \prod_{(\sigma_{j}, k)\in \resset(\varsigma_{0}),
      \Im\sigma_{j}> -l} (\rholog D_{\rholog} - \sigma_{j})\right) w
  \in \rholog^{l}\Hb^{\min (C-l-0,
    1/2-\varsigma_{0}-l-0), \gamma}(\Mlog).
\end{equation*}
Now $\cM$ includes a basis of vector fields in $\mathcal{V}_{b}(\Mlog)$
with the exception of $\pd[\vlog]$,  but $\vlog \pd[\vlog]$ is in
$\cM$.  This leads to:
\begin{lemma}
  \label{lemma:moduletosobolev}
  If $\cMlog^{\ell}w\in\Hb^{p,q}(\Mlog)$, then $\vlog^{\ell}w\in
  \Hb^{p+\ell,q}(\Mlog)$.  More generally, if $\cMlog^{N+\ell}w\in
  \Hb^{p,q}(\Mlog)$, then $\cMlog^{N}\vlog^{\ell}w\in\Hb^{p+\ell,q}(\Mlog)$.
\end{lemma}

\begin{proof}
  Since $D_{\vlog}\vlog\in \cMlog$, we have
  \begin{equation*}
    (\rholog
    D_{\rholog})^{\alpha}D_{\ylog}^{\beta}D_{\vlog}^{\gamma}\vlog^{\ell}\in
    \cMlog ,
  \end{equation*}
  provided $\gamma \leq \ell$, hence by our assumed module regularity,
  \begin{equation*}
    (\rholog D_{\rholog})^{\alpha}D_{\ylog}^{\beta}D_{\vlog}^{\gamma}
    v^{\ell}w \in \Hb^{p,q}(M),
  \end{equation*}
  provided $\alpha + \smallabs{\beta} + \gamma \leq \ell$.
\end{proof}

Applying Lemma~\ref{lemma:moduletosobolev} now yields, for all $N$,
\begin{equation*}
  \cMlog^{N}\vlog^{l}\left( \prod_{(\sigma_{j},k)\in
      \resset(\varsigma_{0}), \Im\sigma_{j} > -l}(\rholog D_{\rholog}
    - \sigma_{j})\right) w \in \rholog^{l}\Hb^{\min (C-0,
    1/2-\varsigma_{0}-0), \gamma}(\Mlog).
\end{equation*}
Since $\rholog^{-l}$ commutes with all generators of $\cMlog$ except
$\rholog D_{\rholog}$ and since $\rholog D_{\rholog}\rholog^{-l} =
\rholog^{-l}\rholog D_{\rholog} + \imath l \rholog^{-l}$, induction on
$N$ shows that we may commute $\rho^{-l}$ through the module factors
to obtain
\begin{equation*}
  \cMlog^{N}\rholog^{-l}\vlog^{l}\left(\prod_{(\sigma_{j}, k) \in
      \resset(\varsigma_{0}), \Im\sigma_{j} > -l}(\rholog D_{\rholog}
    - \sigma_{j}) \right) w \in \Hb^{\min (C-0,
    1/2-\varsigma_{0}-0), \gamma}(\Mlog).
\end{equation*}
In other words, if we set $$\varpi = \rholog / \vlog$$ (ignoring $|\vlog / \rholog| <
1$ for notational convenience) is the defining function of the side
faces $C_{+}$ and $C_{0}$ in the blow-up $[\Mlog ; S]$,
\begin{equation*}
  \cMlog^{N}\varpi^{-l}\left(\prod_{(\sigma_{j}, k) \in
      \resset(\varsigma_{0}), \Im\sigma_{j} > -l}(\rholog D_{\rholog}
    - \sigma_{j}) \right) w \in \Hb^{\min (C-0,
    1/2-\varsigma_{0}-0), \gamma}(\Mlog),
\end{equation*}
which is to say, switching over entirely to coordinates $\varpi=\rholog/\vlog$, $\vlog$,
and $\ylog$ valid in a neighborhood of $C_{+}$ including the corner
$C_{+}\cap \scri$, we finally have the following:
\begin{proposition}\label{proposition:Cplus}
On $C_+,$ uniformly up to the corner $C_+\cap \scri^+$ in $[\Mlog;
S_+],$ $w$ enjoys an asymptotic expansion with powers given by the
resonance index set:
\begin{equation*}
  \left(\prod_{(\sigma_{k},k) \in \resset (\varsigma_{0}),
      \Im\sigma_{j} > - l} (\varpi D_{\varpi} - \sigma_{j})\right) w
  \in \varpi^{l}\Hb^{\infty, *, *}([\Mlog ; S]), 
\end{equation*}
where the $*$'s represent fixed (i.e., independent of $l$) growth
orders.
\end{proposition}

By Proposition~\ref{proposition:doublemellin}, this is now one of the
two ingredients required to prove Theorem~\ref{mainthm} \emph{in the
  short-range case}, giving us the
expansion at $C_+$ uniformly up to $\scri^+.$ To complete the proof of
Theorem~\ref{mainthm} for the short-range case, it thus suffices to obtain the expansion at
$\scri^+,$ with uniform control at $C_+.$

\subsection{Expansion at $\scri^+$: the short-range case}
\label{sec:short-range-case}

In describing asymptotics at $\scri^+,$ \emph{we now specialize to the short-range case} for the sake of clarity of
exposition, before returning to the more general long-range case in the following
section.

Throughout this section we use $R$ to denote the vector field that
lifts to be the radial vector field at $\scri$, i.e.,
\begin{equation*}
  R = \rho D_{\rho} + vD_{v}.
\end{equation*}
We further set $\rad_{k}$ to be the appropriate product of shifted
radial vector fields to test for the $\mathcal{C}^\infty$ index set $0
\equiv \{(-j \imath, 0),\ j \in \NN\}.$  In other
words, we have
\begin{equation*}
  \rad_{k} = \prod_{j=0}^{k}(R + \imath j).
\end{equation*}
Here $\rad_{-1}$ denotes the empty product, i.e., the identity operator.

We begin by treating the short range case, i.e., assuming that $\mass
= 0$; what we will want to prove is that one has a polyhomogeneous
expansion on $[M;S]$.  This means that at the lift of $C_{+}$, which
is still denoted by $C_{+}$, one has an expansion given by the
resonances (this is Proposition~\ref{proposition:Cplus}
above, while at
$\scri$ one has smooth behavior (i.e., the standard expansion), and at
the lift of $C_{0}$ there is rapid decay.  These expansions at
boundary hypersurfaces are supposed to fit together smoothly at the
corners; we recall that by Proposition~\ref{proposition:doublemellin},
the apparent challenge of
verifying matching conditions is moot.

We must thus prove such an expansion at $C_{+}$ (with the resonance
index set), $C_{0}$ (with the empty index set), and $\scri$ (with the
smooth index set).  The $C_{+}$ expansion we have already obtained in
both the short- and long-range cases: this is Proposition~\ref{proposition:Cplus}
above.

At $C_{0}$, the same argument applies, but, due to the support
property of the resonant states (Lemma~\ref{lemma:laurent}), the
coefficients all vanish to infinite order.  In particular, this means
we need not apply the radial factors to obtain the vanishing.  In other
words, for all $l$
\begin{equation}
  \label{eq:c_0-module-reg}
  \cM^{N}w \in \varpi^{l}\Hb^{s,\gamma}(U),
\end{equation}
where $U$ is a neighborhood of $C_{0}$ in $[M;S]$ on which $\varpi$ is
bounded above\footnote{Although $U$ is a subset of the blow-up, we
  abuse notation by treating it as a subset of $\cM$ in defining this
  weighted Sobolev space with a single weight.}
 (say, $\varpi < 1$). Put
differently,
\begin{equation*}
  w \in \varpi^{l}H^{\infty, *, *}(U).
\end{equation*}

We now turn to $\scri$.  In dealing with the expansion near $\scri$ we
consider the $\CI(M)$-submodule $\cMdiff$ of $\cM$ consisting of the first
order differential operators in $\cM$.  Because $\cM$ is generated as
a module over $\Psib^{0}$ by differential operators, regularity with
respect to the module $\cMdiff$ is equivalent to regularity with
respect to $\cM$.  

\begin{lemma}
  \label{lemma:rad-vf-ideal}
  With $\cI$ the ideal of $\CI$ functions vanishing at $S$, one has
  \begin{equation*}
    \left( R + \imath k\right)\cI \subset \cI ( R + \imath(k-1)) + \cI^{2}
  \end{equation*}
  and 
  \begin{equation*}
    [R , \cMdiff] \subset \rho \cMdiff + v
    \cMdiff = \cI \cMdiff.
  \end{equation*}
\end{lemma}

The second part of the lemma is related to the statement that $\cM$
lifts to b-pseudodifferential operators on $[M;S]$, while $\rho
D_{\rho} + vD_{v}$ is the radial vector field associated with the
front face, which has a commutative normal operator.  Thus, its
commutator with anything has an extra order of vanishing (i.e., in
$\rho$ or $v$) at the front face.

\begin{proof}
  First if $a \in \CI(M)$, then
  \begin{equation}
    \label{eq:radial-vf-comm}
    [R, a] = \rho D_{\rho}a + v D_{v}a \in \cI.
  \end{equation}
  Since elements of $\cI$ are of the form $\rho a_{1} + v a_{2}$ with
  $a_{j}\in \CI(M)$, and since
  \begin{equation*}
    (R + \imath)\rho = \rho R, \ (R + \imath)v = vR,
  \end{equation*}
  we have
  \begin{align*}
    (R + \imath k) (\rho a_{1} + va_{2}) 
    &= \rho (R + \imath (k-1))a_{1} + v
    (R + \imath (k-1))a_{2} \\
    &= \rho [R, a_{1}] + \rho a_{1} (R + \imath (k-1)) 
    + v [R , a_{2}] + va_{2} (R+ \imath (k-1)),
  \end{align*}
  with the commutators on the right hand side in $\cI$ as remarked at
  the outset, so the membership of the right hand side in $\cI (R +
  \imath (k-1)) + \cI^{2}$ follows.

  Turning to $[R, \cMdiff]$, using~\eqref{eq:radial-vf-comm} again, it
  suffices to show for a set of generators $V_{j}$ of $\cMdiff$ that
  $[R, V_{j}] \in \cI\cMdiff$.  Using $\rho D_{\rho}$, $vD_{v}$, $\rho
  D_{v}$, and $D_{y_{j}}$ in local coordinates as generators, all
  commute with $R$ since they are homogeneous of
  degree zero under the action of dilations $(\rho, v, y) \to (t\rho,
  tv, y)$, $t>0$, in the first two variables.
\end{proof}

In fact, more generally one has
\begin{lemma}
  \label{lemma:rad-factors-through-module}
  With $\cI$ as above:
  \begin{align*}
    [\rad_{k}, \cMdiff^{l}] &\subset
    \sum_{j=0}^{k}\cI^{j+1}\cMdiff^{l}\rad_{k-1-j} \\
    &\subset \sum_{j=0}^{k+1}\Psib^{-j-1}\cMdiff^{l+j+1}\rad_{k-1-j}
  \end{align*}
\end{lemma}
Here the vanishing factor of powers of $\cI$ arises from the
classicality of the coefficients, so if one has logarithmic
coefficients, one needs additional factors of the radial vector field
plus appropriate constants.

\begin{proof}
  The second inclusion in the statement of the lemma follows from $\cI
  \subset \Psib^{-1} \cMdiff$, which we prove in
  Lemma~\ref{lemma:inclusion}.

  First consider $k=0$, i.e., $[R, V_{1}\dots V_{l}]$ with $V_{j} \in
  \cMdiff$.  This is of the form
  \begin{equation*}
    [R, V_{1}]V_{2}\dots V_{l} + V_{1}[R, V_{2}]V_{3}\dots V_{l} + \dots + V_{1}\dots
    V_{l-1}[R, V_{l}],
  \end{equation*}
  and the commutators are in $\cI \cMdiff$ by the second half of
  Lemma~\ref{lemma:rad-vf-ideal}.  Now, as $[\cMdiff, \cI]\subset
  \cI$, one can commute the $\cI$ factors to the front iteratively.
  This proves the $k=0$ case, namely that $[R, V_{1}\dots
  V_{l}]\subset \cI \cMdiff^{l}$.

  Now suppose $k\geq 1$, and that the lemma has been proved with $k$
  replaced by $k-1$.  Then
  \begin{align}
    \label{eq:higher-k-comm}
    [\rad_{k}, \cMdiff^{l}] 
    \subset (R + \imath k) [\rad_{k-1},\cMdiff^{l}] + [R,
    \cMdiff^{l}]\rad_{k-1}
  \end{align}
  By the inductive hypothesis the first term on the right hand side is
  in
  \begin{align*}
    (R + \imath k)\sum_{j=0}^{k-1}\cI^{j+1}\cMdiff^{l}\rad_{k-2-j}
  \end{align*}
  Commuting $R + \imath k$ through the ideal
  factors using the first half of Lemma~\ref{lemma:rad-vf-ideal}
  iteratively, this itself lies in
  \begin{align*}
    \sum_{j=0}^{k-1}\left( \cI^{j+1}(R + \imath
      (k-j-1))\cMdiff^{l}\rad_{k-2-j} +
      \cI^{j+2}\cMdiff^{l}\rad_{k-2-j}\right)
  \end{align*}
  By the $k=0$ case, commuting $(R + \imath j)$
  factors on the left of $\cMdiff^{l}$ to the right gives commutators
  in $\cI \cMdiff^{l}$, so this expression is in
  \begin{align*}
    \sum_{j=0}^{k-1}\left( \cI^{j+1}\cMdiff^{l}\rad_{k-1-j} +
      \cI^{j+2}\cMdiff^{l}\rad_{k-2-j}\right) 
    &\subset \sum_{j=1}^{k}\cI^{j+1}\cMdiff^{l}\rad_{k-1-j}.
  \end{align*}
  which is of the form in the statement of the lemma.  On the other
  hand, the second term in~\eqref{eq:higher-k-comm} is
  \begin{equation*}
    [R, \cMdiff^{l}]\rad_{k-1},
  \end{equation*}
  so by the $k=0$ case we get
  \begin{equation*}
    \cI\cMdiff^{l} \rad_{k-1}
  \end{equation*}
  for this term, which is of the form given in the last term in the
  statement of the lemma.
\end{proof}

The main claim is:

\begin{proposition}
  \label{prop:ff-gain-in-diff}
  If $w \in \Hb^{s,\gamma}(M)$ with $Lw \in \dCI(M)$, we have
  \begin{equation*}
    \cMdiff^{N}\rad_{k}w \in \Hb^{s+(k+1), \gamma}(M).
  \end{equation*}
\end{proposition}

Notice that this proposition improves the b-regularity, but not the
decay; in particular, this does \emph{not} involve normal operators.
However, once we have this, we can use the infinite order vanishing at
$C_{0}$ to establish vanishing at the front face, as we show below.

\begin{proof}
  The result follows from Proposition~\ref{prop:radial-b-estimate} if there are no radial vector
  factors (so $k=-1$).  If $k=0$, notice that
  \begin{equation*}
    L + 4D_{v} (v D_{v} + \rho D_{\rho}) \in \cMdiff^{2},
  \end{equation*}
  so $Lw \in \dCI$ and $\cMdiff^{N}w \in \Hb^{s,\gamma}$ for all $N$
  implies that $D_{v}R w \in \Hb^{s,\gamma}$ by~\eqref{eq:Lmodule}.
  Because $D_{v}$ is elliptic on $\WFb(w)$, this yields $\rad_{0} w
  \in \Hb^{s+1, \gamma}(M)$.  To finish the $k=0$ case, we now rewrite
  $D_{v}\cMdiff^{N}\rad_{0}w$ by commuting $D_{v}$ with $\cMdiff$.  In
  particular, it suffices to consider the usual set of generators for
  $\cMdiff$; the only one not commuting with $D_{v}$ is $vD_{v}$, but
  $D_{v}(v D_{v}) = (D_{v} v)D_{v} \in \cMdiff D_{v}$, so $D_{v}
  \cMdiff \subset \cMdiff D_{v} + \cMdiff$.  Thus, iterating we find that
  \begin{equation*}
    D_{v}\cMdiff^{N} \subset \cMdiff^{N}D_{v} + \cMdiff^{N}.
  \end{equation*}
  Consequently, we obtain
  \begin{align*}
    D_{v} \cMdiff^{N}R w &\subset
    \cMdiff^{N}D_{v}Rw + \cMdiff^{N}Rw \\
    &\subset \cMdiff^{N} Lw  + \cMdiff^{N+2}w \subset \Hb^{s, \gamma}(M).
  \end{align*}
  The ellipticity of $D_{v}$ on $\WFb (u)$ now proves the $k=0$ case
  of the proposition.

  Now suppose $k \geq 1$, and that the proposition has been proved
  with $k$ replaced by $k-1$.  We use then that
  \begin{align*}
    D_{v} \rad_{k} = \rad_{k-1}D_{v} R,
  \end{align*}
  so
  \begin{align*}
    D_{v} \rad_{k} &\in \rad_{k-1}L + \rad_{k-1}\cMdiff^{2} \subset
    \rad_{k-1}L + \sum_{j=0}^{k}\Psib^{-j}\cMdiff^{j+2}\rad_{k-1-j},
  \end{align*}
  where we applied Lemma~\ref{lemma:rad-factors-through-module} for
  the last inclusion.  Thus, using the inductive hypothesis,
  \begin{equation*}
    D_{v}\rad_{k} w \in \Hb^{s+k,\gamma}.
  \end{equation*}
  Again, as $D_{v}$ is elliptic in the microlocally relevant region,
  \begin{equation*}
    \rad_{k}w \in \Hb^{s+(k+1),\gamma}.
  \end{equation*}
  A similar result holds even with a factor $\cMdiff^{N}$ added, by
  the same argument as in the $k=0$ case, which completes the proof of
  the proposition.
\end{proof}

We now use the proposition, which as pointed out gives additional
regularity without additional decay, to prove vanishing at the front
face using the infinite order vanishing at $C_{0}$.  First, fixing
$v_{0} < 0$, we already have $O(\rho^{\infty})$ bounds for $w$ near
$v_{0}$.  Further, we have the following estimate near $C_{0}$:
\begin{lemma}
  \label{lemma:ff-improvement}
  Let $U$ be a neighborhood of $\overline{C_{0}}$ in $[M;S]$ as above.  Then for
  any $\epsilon > 0$ and $N,N' \in \NN$,
  \begin{equation}
    \label{eq:diff-ff-reg-with-weight}
    D_{v}^{k+1}\cMdiff^{N}\rad_{k}w \in (\rho / v)^{N'}v^{-\epsilon}\Hb^{s,\gamma}(U).
  \end{equation}
\end{lemma}

\begin{proof}
  Without the $(\rho / v)^{N'}$ or $v^{-\epsilon}$ factors, the
  desired estimate is just the regularity statement of
  Proposition~\ref{prop:ff-gain-in-diff}.  On the other hand, since
  $\rad_{k}\in \cMdiff^{k+1}$, the decay statement~\eqref{eq:c_0-module-reg}
  yields the growth/decay statement
  \begin{equation*}
    \cMdiff^{\tilde{N}}\rad_{k}w \in (\rho / v)^{N'}\Hb^{s,\gamma}(U).
  \end{equation*}
  As $v^{k+1}D_{v}^{k+1}\in \cMdiff^{k+1}$, we then have
  \begin{equation*}
    D_{v}^{k+1}\cMdiff^{N}\rad_{k}w \in v^{-k-1}(\rho/v)^{N'}\Hb^{s,\gamma}(U).
  \end{equation*}
  Fixing $k$, taking $N'$ large, and interpolating with
  Proposition~\ref{prop:ff-gain-in-diff} completes the proof.
\end{proof}

Now integrating~\eqref{eq:diff-ff-reg-with-weight} $k+1$ times in $v$
from $v_{0}$ gives
\begin{equation*}
  \cMdiff^{N}\rad_{k}w \in v^{-\epsilon}(\rho^{N'} + \rho^{N'}v^{-N'+k+1})\Hb^{s,\gamma}(U),
\end{equation*}
which is, with $N' > k+1$, an order $k+1 - \epsilon$ vanishing
statement at the front face in the region $U$ where $|\rho / v|$ is
bounded.  

Having this decay in $U$, we can proceed further into the front face.
Since $w$ has no b-wavefront set except at $\rho=v=0,$ it is in
particular smooth in $v,$ and we can rewrite $\cM \rad_{k} w$ as an
iterated integral of its $(k+1)$-st derivative in $v$.  Integrating
from, say, $v = - \rho /2$, this gives an estimate $$\cMdiff^{N}
\rad_{k} w \in (v +
C\rho)^{k+1}\Hb^{s,\gamma}(M),$$ with $v + C \rho$ being
the length of the integration curve in the coordinates $v,y,\
\varpi\equiv \rho / v$ valid in a neighborhood of the interior of $\scri^+$.
Now we lift the module regularity statement on $M$ to the blowup:
since the generators of the module span a basis of b-vector fields on
$[M; S],$ the module regularity lifts to give $\Hb^\infty$ regularity
on the blowup (as the generators of the module lift to nondegenerate
b-vector fields on the blow-up), i.e., 
the module regularity means that $$\rad_{k}w \in (v + C\rho)^{k+1}
\Hb^{N, *, *}([M ; S])$$ for suitable fixed (i.e., $k$-independent)
weights $*$.  As $v + C\rho$ defines the front face in the relevant
region, this is exactly the desired polyhomogeneity statement at
$\scri$.  

This finishes the proof of Theorem~\ref{mainthm} in the
short-range case.

\begin{remark}
  \label{rem:price}
  In~\cite[Section 10.1]{radfielddecay}, we incorrectly stated that
  the radiation field was rapidly decaying as it is in Minkowski
  space.  Instead, we have a form of Price's law, which in this case
  states that the radiation field decays as $s^{-\frac{n-1}{2}-1}$.
  In $3+1$ dimensions, this means that the radiation field is expected
  to decay as $s^{-2}$ and the solution of the wave equation should
  decay as $t^{-3}$ in the interior of the light cones.
\end{remark}

\subsection{Expansion at $\scri^+$: the long-range case}
\label{sec:long-range-case}

In this section we return to the general setting $m\neq 0$.  The
logification introduced in Section~\ref{sec:logification} added
logarithmic terms to the operators in question (in order to remove
them from the geometry).  We proceed in much the same way as in the
previous section, though significant modifications arise from the
presence of $\rholog\log\rholog$ terms.  In particular, the main
difference is that, while in the short range case, we showed that $w$
was polyhomogeneous at $\scri$ with index set
\begin{equation*}
  \smoothset = \{ (-\imath k, 0) : k = 0, 1 , 2, \dots\},
\end{equation*}
in the long-range setting we show that, owing to the additional log
terms in the coefficients of $L$, $w$ is polyhomogeneous at $\scri$
with index set
\begin{equation*}
  \logset = \{ (-\imath k, j) : k = 0, 1, 2, \dots, j = 0, 1, \dots , 2k\}.
\end{equation*}

In what follows, we will abuse notation by letting $w$ denote $\taut_*
w,$ its pushforward from $M$ to $\Mlog.$

Let $\rad_{k}$ be given by the following product of radial vector
fields (note that this differs from the product in
Section~\ref{sec:short-range-case}):
\begin{equation*}
  \rad_{k} = \prod_{j=0}^{k}(\rholog D_{\rholog} + \vlog D_{\vlog} +
  \imath j)^{2j+1}
\end{equation*}
On $[\Mlog ; S]$ with coordinates $\rholog,\ \vlog,\ \varpi,$ observe that $\rad_{k}$ has the following form:
\begin{equation*}
  \rad_{k} = \prod _{j=0}^{k}(\rholog D_{\rholog} + \imath j)^{2j+1}
\end{equation*}
In other words, $\rad_{k}$ is the appropriate product of radial
vector fields at $\scri$ to test for polyhomogeneity with index set
$\logset$.  For convenience with our bookkeeping, we also define the
$k$-th triangular number as follows:
\begin{equation*}
  t_{-1} = 0, \quad t_{k} = t_{k-1} + k
\end{equation*}

As in the short range case, the support property of the resonant
states means that all coefficients vanish to infinite order at
$C_{0}$, so that for all $\ell$, we have
\begin{equation*}
  w \in \varpi^{\ell}\Hb^{\infty, *, *}(U),
\end{equation*}
where $U$ is a neighborhood of $C_{0}$ in $[\Mlog ; S]$ on which $\varpi$
is bounded above.

The main difference in the proof concerns the behavior at $\scri$.  In
the previous section, the crux of the proof was
Proposition~\ref{prop:ff-gain-in-diff}. The replacement for this
proposition is the following:
\begin{proposition}
  \label{prop:ff-gain-in-diff-log}
  If $w \in \Hb^{s,\gamma}(\Mlog)$ with $Lw \in \dCI(M)$, we have
  \begin{equation*}
    \cMdifflog ^{N}\rad_{k} w \in \Hb^{s+(k+1), \gamma - 0}(\Mlog).
  \end{equation*}
\end{proposition}

We defer for now a discussion of the proof of
Proposition~\ref{prop:ff-gain-in-diff-log} and note that the following
analogue of Lemma~\ref{lemma:ff-improvement} immediately follows (with
the same proof):
\begin{lemma}
  \label{lemma:ff-improvement-log}
  Let $U$ be a neighborhood of $C_{0}$ in $[\Mlog ; S]$ as in the
  discussion immediately preceding equation~\eqref{eq:c_0-module-reg}.
  Then for any $\epsilon > 0$ and $N'\in \NN$,
  \begin{equation}
    \label{eq:inclusion-log}
    D_{\vlog}^{k+1}\cMdifflog\rad_{k} w \in (\rholog / \vlog)^{N'}\vlog^{-\epsilon}\Hb^{s,\gamma-0}(U).
  \end{equation}
\end{lemma}

As in the previous section, we can then integrate $k+1$ times in
$\vlog$ to obtain the desired vanishing (and hence polyhomogeneity)
statements at $\scri$.  This completes the proof of
Theorem~\ref{mainthm} in the long-range case.

We now turn our attention to the proof of
Proposition~\ref{prop:ff-gain-in-diff-log}.  Suppose we are able to
prove the following lemma (which is the analogue of
Lemma~\ref{lemma:rad-factors-through-module}):

\begin{lemma}
  \label{lemma:iterative-step-log}
  If $w \in \Hb^{s,\gamma}$ is as above and $Lw \in \dCI(\Mlog)$, then
  \begin{align*}
    D_{\vlog}\rad_{k} w \in \sum_{j=0}^{k}\cIlog^{j} \cMdifflog^{1 +
      2(t_{k}-t_{k-1-j})}\badmod^{2}\rad_{k-1-j}w + \dCI(\Mlog).
  \end{align*}
\end{lemma}

\begin{proof}[Proof of Proposition~\ref{prop:ff-gain-in-diff-log}]
  The proposition holds if there are no factors of the radial vector
  field (i.e., if $k = -1$) by propagation of singularities
(Proposition~\ref{prop:radial-b-estimate}).  If
  $k=0$, we notice that
  \begin{equation*}
    L + 4 D_{\vlog} \rad_{0} \in \badmod^{2},
  \end{equation*}
  so because $Lw \in \dCI$ and $\badmod^{2} w \in \Hb^{s,\gamma-0}$
  (by \eqref{badmodsobolev}),
  we have $D_{\vlog} \rad_{0} \in \Hb^{s,\gamma-0}$.  Because
  $D_{\vlog}$ is elliptic on $\WFb(w)$, this yields $\rad_{0} w \in
  \Hb^{s+1,\gamma - 0}$.  

  To finish the $k=0$ case, we now rewrite
  $D_{\vlog}\cMdifflog^{N}\rad_{0}w$ by commuting $D_{\vlog}$ with
  $\cMdifflog$.  In particular, $D_{\vlog}\cMdifflog
  \subset \cMdifflog D_{\vlog} + \cMdifflog$ and so
  \begin{equation*}
    D_{\vlog}\cMdifflog^{N} \subset \cMdifflog^{N}D_{\vlog} + \cMdifflog^{N}.
  \end{equation*}
  We thus obtain
  \begin{align*}
    D_{v}\cMdifflog^{N}\rad_{0} w &\subset \cMdifflog^{N} D_{\vlog}\rad_{0}
    w + \cMdifflog^{N}\rad_{0}w \\
    &\subset \cMdifflog^{N}Lw +
    \cMdifflog^{N}\badmod^{2}w \subset \Hb^{s,\gamma-0}(\Mlog).
  \end{align*}
  The ellipticity of $D_{\vlog}$ on $\WFb(w)$ now proves the $k=0$
  case of the proposition.

  Now suppose $k\geq 1$ and that the proposition has been proved with
  $k$ replaced by $k-1$.  We then use
  Lemma~\ref{lemma:iterative-step-log} to see that
  \begin{align*}
    D_{\vlog} \rad_{k}w &\in \sum_{j=0}^{k}\cIlog^{j}\cMdifflog^{1 +
      2(t_{k}-t_{k-1-j})}\badmod^{2}\rad_{k-1-j}w \\
    &\subset
    \sum_{j=0}^{k}\Psib^{-j}\cMdifflog^{1+j+2(t_{k}-t_{k-1-j})}\badmod^{2}\rad_{k-1-j}w
    \in \Hb^{s+k,\gamma-0}(\Mlog)
  \end{align*}
  by the induction hypothesis.  Because $D_{\vlog}$ is elliptic in the
  microlocally relevant region, we see that $\rad_{k}w \in
  \Hb^{s+(k+1),\gamma-0}(\Mlog)$.

  A similar result holds even with a factor $\cMdifflog^{N}$ added, by
  the same argument as in the $k=0$ case (and using the fact that $\cMdifflog$
  preserves $\CIlog$); this completes the proof of
  the proposition.
\end{proof}

We now turn our attention to the proof of
Lemma~\ref{lemma:iterative-step-log}.  The intuitive idea behind the
proof is as before, namely that commuting the radial vector field
through the various factors yields an improvement.  Unfortunately, it
is a bit more complicated than in the short-range case: 
\begin{lemma}
  \label{lemma:log-commutators1}
  Let $R = \rholog D_{\rholog} + \vlog D_{\vlog}$.  The following
  relations hold:
  \begin{enumerate}
    \item $(R + \imath k)\badmod \subset \badmod (R + \imath k) + \cMdifflog$
    \item $\badmod \cMdifflog \subset \cMdifflog \badmod$
    \item $(R+\imath k) \cMdifflog \subset \cMdifflog (R + \imath k) +
    \cIlog \cMdifflog$
    \item $\cMdifflog \cIlog \subset \cIlog \cMdifflog$
    \item $(R + \imath k)\cIlog \subset \cIlog (R + \imath (k-1)) + \rholog
    \CIlog + \cIlog^{2}$
    \item $(R + \imath k)\rholog \CIlog \subset \rholog \CIlog (R+\imath
    (k-1)) + \rholog \cIlog$
  \end{enumerate}
\end{lemma}

\begin{proof}
  We first observe that if $a\in \cIlog$, then $[R, a] = \rholog
  D_{\rholog}a + \vlog D_{\vlog} a \in \cIlog$.  

  Now we observe that $[R, \rholog \log\rholog D_{\rholog}] =
  \frac{1}{\imath}\rholog D_{\rholog} \in \cMdifflog$ and $[R, \rholog
  \log \rholog D_{\vlog}] = \frac{1}{\imath}\rholog D_{\vlog}\in
  \cMdifflog$.  Any element of $\badmod$ can be written as $V +
  a_{1}\rholog\log\rholog D_{\rholog} + a_{2} \rholog\log\rholog
  D_{\vlog}$, where $V \in \cMdifflog$ and $a_{i}\in \CIlog$, proving
  the first statement.  

  The second statement follows from the observation that $[\cMdifflog
  , \badmod] \subset \badmod$.

  The proof of Lemma~\ref{lemma:rad-vf-ideal}, together with the
  observation that $[R, a] \in \cIlog$ shows that the third
  statement holds.

  The fourth statement follows from the observation that $[\cMdifflog,
  \cIlog] \subset \cIlog$.

  For the fifth statement, we compute.  The proof of
  Lemma~\ref{lemma:rad-vf-ideal} shows that the statement is true for
  elements of $\cIlog$ of the form $a_{1} \rholog + a_{2}\vlog$, so we must
  only show it for elements of the form $a \rholog\log\rholog$, where
  $a \in \CIlog$.  We then compute
  \begin{align*}
    (R+\imath k) a \rholog \log \rholog  = a \rholog \log \rholog
    (R + \imath (k-1)) + \frac{1}{\imath}\rho a + \rholog \log
    \rholog (R a) ,
  \end{align*}
  which lies in the desired space.

  The final statement is similar.  Suppose $a\in \CIlog$, then
  \begin{align*}
    (R+\imath k)\rho a = \rho a (R + \imath (k-1)) + \rho (Ra) .
  \end{align*}
\end{proof}

By repeatedly applying Lemma~\ref{lemma:log-commutators1}, we obtain
the following iterative version of the lemma:
\begin{lemma}
  \label{lemma:log-commutators-2}
  Suppose $\alpha$, $\beta$, $\gamma$, $\delta$, and $\epsilon$ are
  integers, and that $\gamma \geq 1$.  Let $R = \rholog D_{\rholog} +
  \vlog D_{\vlog}$ and let $\tilde{R}^{j}$ denote any product of $j$
  shifts of the radial vector field $R$.  We then have that
  \begin{itemize}
  \item $\displaystyle (R + \imath k )^{\epsilon}\rho^{\alpha}
    \subset \rho ^{\alpha}(R + \imath (k-\alpha))^{\epsilon} +
    \sum_{i=1}^{\epsilon}\sum_{a = 0}^{\min (i, \epsilon -
      i)}\rho^{\alpha + a}\cIlog^{i -
      a}\tilde{R}^{\epsilon - i - a} $
  \item $\!
    \begin{aligned}[t]
      (R+\imath k)^{\epsilon}\cIlog^{\beta} &\subset \sum_{a =
      0}^{\min(\epsilon, \beta)}\rho^{a}\cIlog^{\beta - a}(R + \imath
    (k-\beta))^{\epsilon - a} \\
    &\quad + \sum_{i = 1}^{\epsilon}
    \sum_{a=0}^{\min(\epsilon - i, \beta + i)}\rho^{a}\cIlog^{\beta +
      i - a}\tilde{R}^{\epsilon - i - a} 
    \end{aligned}$

  \item $
    \begin{aligned}[t]
      (R+ \imath k)^{\epsilon}\cMdifflog^{\gamma}\badmod^{\delta}
      &\subset \sum_{d=0}^{\min(\delta, \epsilon)}\cMdifflog^{\gamma +
        d}\badmod^{\delta - d}(R+ \imath k)^{\epsilon - d} \\
      &\quad + \sum_{d=0}^{\min(\delta,
        \epsilon)}\sum_{i=1}^{\epsilon-d}\sum_{a = 0}^{\min(i, \epsilon
        - d - i)} \rho^{a}\cIlog^{i - a}\cMdifflog^{\gamma +
        d}\badmod^{\delta - d}\tilde{R}^{\epsilon - d - i - a}
    \end{aligned}$
  \end{itemize}
  \begin{align*}
  \end{align*}
\end{lemma}

\begin{remark}
  Lemma~\ref{lemma:log-commutators1} implies that $(R+\imath
  k)^{\epsilon}\rholog^{\alpha}
  \cIlog^{\beta}\cMdifflog^{\gamma}\badmod^{\delta}$ is contained in a
  sum of terms of the form
  \begin{equation*}
    \rholog^{a}\cIlog^{b}\cMdifflog^{c}\badmod^{d}\tilde{R}^{e},
  \end{equation*}
  where all exponents are nonnegative, $\gamma + \delta=
  c + d$, and $2\alpha + \beta + 2\gamma + \delta + \epsilon = 2a + b
  + 2c + d + e$.  The leading terms are those with $a + b = \alpha +
  \beta$.
\end{remark}

\begin{proof}
  The main idea that one can ``spend'' a power of $(R + \imath k)$ to
  do one of the following:
  \begin{itemize}
  \item Turn a factor of $\badmod$ into a factor of $\cMdifflog$,
  \item Turn a factor of $\cMdifflog$ into a factor of $\cIlog$,
  \item Turn a factor of $\cIlog$ into a factor of $\rho \CIlog +
    \cIlog^{2}$, or
  \item Turn a factor of $\rho$ into a factor of $\rho \cIlog$.
  \end{itemize}
  Moreover, commuting the radial vector field through a power of
  $\rho$ or $\cIlog$ shifts it by $\imath$.

  We show only the easiest of the three cases to indicate the method
  of proof.  

  By applying Lemma~\ref{lemma:log-commutators1} repeatedly, we see that
  \begin{equation*}
    (R+\imath k)\rho^{\alpha} \subset \rho^{\alpha}(R + \imath
    (k-\alpha)) + \rho^{\alpha}\cIlog.
  \end{equation*}
  Now suppose that we have shown the first statement for $\epsilon$.
  We have
  \begin{align*}
    (R + \imath k)^{\epsilon+1} \rho^{\alpha} &\subset (R+ \imath
    k)\rho^{\alpha}(R + \imath (k-\alpha))^{\epsilon} \\
    &\quad + (R + \imath
    k)\sum_{i=1}^{\epsilon}\sum_{a = 0}^{\min(i, \epsilon -
      i)}\rho^{\alpha + a}\cIlog^{i-a}\tilde{R}^{\epsilon-i-a} \\
    &\subset \rho^{\alpha}(R + \imath (k-\alpha))^{\epsilon + 1} +
    \rho^{\alpha}\cIlog (R + \imath (k-\alpha))^{\epsilon} \\
    &\quad + \sum_{i=1}^{\epsilon} \sum_{a = 0}^{\min (i, \epsilon -
      i)}\rho^{\alpha + a}\cIlog^{i-a}\tilde{R}^{\epsilon + 1 - i - a}
    \\
    &\quad + \sum_{i=1}^{\epsilon}\sum_{a=0}^{\min(i, \epsilon
      -i)}\left(\rho^{\alpha+a+1}\cIlog^{i-a-1} + \rho^{\alpha +
        a}\cIlog^{i+1} \right)\tilde{R}^{\epsilon-i-a} \\
    &\subset \rho^{\alpha}(R + \imath (k - \alpha))^{\epsilon+1} +
    \sum_{i=1}^{\epsilon+1}\sum_{a=0}^{\min(i,\epsilon+1-i)}\rho^{\alpha+a}\cIlog^{i
    - a}\tilde{R}^{\epsilon+1-i-a},
  \end{align*}
  as desired.
\end{proof}

Putting Lemma~\ref{lemma:log-commutators-2} together, we have the following:
\begin{lemma}
  \label{log-commutators3}
  Again suppose that $\alpha$, $\beta$, $\gamma$, $\delta$, and
  $\epsilon$ are natural numbers with $\gamma \geq 1$.  Then
  \begin{align*}
    (R+ \imath k)^{\epsilon}
    \rho^{\alpha}\cIlog^{\beta}\cMdifflog^{\gamma}\badmod^{\delta}
    &\subset \sum_{d=0}^{\min(\delta, \epsilon)}\sum_{a =0}^{\min
      (\beta, \epsilon - d)}\rho^{\alpha + a} \cIlog^{\beta -
      a}\cMdifflog^{\gamma + d}\badmod^{\delta - d}(R +
    \imath(k-\alpha-\beta))^{\epsilon-d-a} \\
    &\quad +\sum_{d=0}^{\min(\delta, \epsilon)}\sum_{i=1}^{\epsilon -
      d}\sum_{a=0}^{\min(\beta +i, \epsilon -
      d-i)}\rho^{\alpha+a}\cIlog^{\beta+i-a}\cMdifflog^{\gamma+d}\badmod^{\delta-d}\tilde{R}^{\epsilon-d-i-a} .
  \end{align*}
  In particular, one has, for $\epsilon \geq \beta + 2$,
  \begin{align*}
    (R + \imath k)^{\epsilon}\cMdifflog^{\gamma} \badmod^{2} &\subset
    \cMdifflog^{\gamma + 2}\badmod^{2}(R + \imath k )^{\epsilon - 2}
    \\ 
    &\quad + \cIlog \cMdifflog^{\gamma + \epsilon - 1}\badmod^{2} \\
    (R + \imath
    k)^{\epsilon}\cIlog^{\beta}\cMdifflog^{\gamma}\badmod^{2} &\subset
    \cIlog^{\beta}\cMdifflog^{\gamma + \beta + 2}\badmod^{2}(R +
    \imath (k-\beta))^{\epsilon - 2-\beta} \\
    &\quad + \cIlog^{\beta +
      1}\cMdifflog^{\gamma + \epsilon - 1}\badmod^{2}
  \end{align*}
\end{lemma}

\begin{proof}
  The proof of the first statement merely combines the three
  statements of Lemma~\ref{lemma:log-commutators-2}, while the second
  statement follows from the observations that $\rho \in \cIlog$
  and $\tilde{R} \in \cMdifflog$.
\end{proof}

\begin{proof}[Proof of Lemma~\ref{lemma:iterative-step-log}]
  We proceed via induction on $k$.  Lemma~\ref{lemma:Lnormop}
  establishes that $L + 4 D_{\vlog}\rad_{0} \in \badmod^{2} \subset
  \cMdifflog \badmod^{2}$, finishing the $k=0$ case of the lemma.

  We now suppose the lemma is true with $k$ replaced by $k-1$.  For
  convenience, we let $$s(k,j) = 1 + 2(t_{k}- t_{k-1-j})$$ and observe that
  \begin{align*}
    D_{\vlog}\rad_{k} w &= (R+\imath (k-1))^{2k+1}D_{\vlog}\rad_{k-1}w
    \\
    &\in (R+ \imath(k-1))^{2k+1}
    \sum_{j=0}^{k-1}\cIlog^{j}\cMdifflog^{s(k-1,j)}\badmod^{2}\rad_{k-2-j}w
    \\
    &\subset
    \sum_{j=0}^{k-1}\cIlog^{j}\cMdifflog^{s(k-1,j)+j+2}\badmod^{2}(R +
    \imath(k-1-j))^{2k-1-j}\rad_{k-2-j}w \\
    &\quad +
    \sum_{j=0}^{k-1}\cIlog^{j+1}\cMdifflog^{s(k-1,j)+2k}\badmod^{2}\rad_{k-2-j}w ,
  \end{align*}
  where for the second inclusion we applied
  Lemma~\ref{log-commutators3}.  Because $R \in \cMdifflog$, we then
  have
  \begin{align*}
    D_{\vlog}\rad_{k}w &\in
    \sum_{j=0}^{k-1}\cIlog^{j}\cMdifflog^{s(k-1,j)+2j+2}\badmod^{2}\rad_{k-1-j}w
    \\
    &\quad + \sum_{j=1}^{k}\cIlog^{j}\cMdifflog^{s(k-1,j-1)+
      2k}\badmod^{2}\rad_{k-1-j}w.
  \end{align*}
  We finally note that $s(k-1,j) + 2j+2 = 1 + 2 t_{k-1} - 2t_{k-2-j} +
  2j+2 = 1 + 2(t_{k}-t_{k-1-j})$ and $s(k-1,j-1) + 2k = 1 +
  2(t_{k}-t_{k-1-j})$, finishing the proof.
\end{proof}

\appendix \section{The Kerr metric}
\label{appendix:Kerr}

In this appendix, we discuss the Kerr metric near null infinity as an
example of a Lorentzian scattering metric.

The Kerr metric (with our ``mostly-minus'' sign convention) can be written
\begin{align*}
  \left( 1 - \frac{2Mr}{\Sigma}\right)\,dt^{2} +
  &\frac{4Mar\sin^{2}\theta}{\Sigma}\,dt\,d\varphi -
  \frac{\Sigma}{\Delta}\,dr^{2} - \Sigma \,d\theta^{2} - \left( r^{2} + a^{2} +
    \frac{2Ma^{2}r\sin^{2}\theta}{\Sigma}\right)\sin^{2}\theta\,d\varphi^{2}
  , \\
  \Sigma &= r^{2} + a^{2} \cos^{2}\theta, \\
  \Delta &= r^{2} - 2Mr + a^{2}.
\end{align*}
We now introduce the new variables $$\rho = \frac 1t,\ \vv = 2\left(1-\frac rt \right)$$
so that that the cone $r=t$ now becomes $\vv=0.$  We easily compute
\begin{equation}\label{kerr1}
\frac{\Sigma}{\Delta} \sim 1 + 2M \rho + M \rho \vv,
\end{equation}
\begin{equation}\label{kerr2}
\frac{r}{\Sigma} \sim \rho + \frac{\rho \vv}2
\end{equation}
where we use the notation $f \sim g$ if $f-g =O (\rho^2)+O(\rho
\vv^2)$ near $\rho=\vv=0$ and we will write $\err$ for terms that are
$O(\rho^2) + O(\rho \vv^2)$ below.
Thus we may write
\begin{multline}
g= \big( 1-2M (\rho+\rho \vv/2+\err)\big) \frac{d\rho^2}{\rho^4}- 2 M a
r\sin^2\theta (\rho+\rho \vv/2 + \err)\big(\frac{d\rho}{\rho^2}
d\varphi+d\varphi\frac{d\rho}{\rho^2}\big)
\\ -\big(1 + 2M \rho + M \rho \vv\big)
\big(-(1-\vv/2)\frac{d\rho}{\rho^2}-\frac 12
\frac{d\vv}{\rho}\big)^2-\Sigma d\theta^2-(*) d\varphi^2.
\end{multline}
We then compute the coefficient of $(d\rho/\rho^2)^2$ as
$$
g_{\rho\rho} = \vv-4M\rho-\vv^2/4+\err.
$$
Meanwhile, the coefficient of $(d\rho dv+dv d\rho)/\rho^3$ is given by
$$
\frac 12 \big(1-\frac v2 )(1+2M \rho + M\rho v+\err\big)= \frac 12 +O(\rho)+O(v),
$$
while all other cross terms with $d\rho$ are of the form $O(\rho^{-1})
d\rho \, d\bullet,$ with $\bullet=\theta,\varphi,$ or $v.$
Thus, setting $m=4M,$ and changing coordinates to $$v\equiv \vv-\vv^2/4$$
near $\vv=0$ brings the metric to the desired form.

Meanwhile, we continue to compute in the variables $\rho, \vv$ for the
moment.  The dual Kerr metric has the form
\begin{multline}
\frac{1}{\Delta}\big( r^2 +a^2+ \frac{2 M a^2 r}{\Sigma} \sin^2
\theta)\pa_t^2+ \frac{2 M r}{\Sigma}\frac{a}{\Delta}
(\pa_\varphi\pa_t+\pa_t
\pa_\varphi)-\frac{1}{\Delta\sin^2\theta}\big(1-\frac{2M
  r}{\Sigma}\big) \pa_\varphi^2 -\frac{\Delta}{\Sigma}\pa_r^2 -\frac{1}{\Sigma}\pa_\theta^2.
\end{multline}

Changing coordinates from $r,t$ to $\rho,\vv$ gives,
$$
\pa_t \leadsto -\rho^2\pa_\rho+ 2(1-\vv/2) \rho \pa_{\vv},\quad \pa_r
\leadsto -2 \rho\pa_{\vv}.
$$
Thus, using \eqref{kerr1}, \eqref{kerr2}, the $r,t$ block of the metric can be rewritten in coordinates
$\rho,v,\theta,\varphi$ as
\begin{align*}
&(1+O(\rho)) \pa_t^2-(1+O(\rho))
\pa_r^2 \\
&=(1+O(\rho))( -\rho^2\pa_\rho+ 2(1-\vv/2) \rho \pa_{\vv})^2-(1+O(\rho)) (-2
\rho \pa_{\vv})^2.
\end{align*}
In other words, the scattering principal symbol associated to these
terms (with canonical dual variables $\xi,\ggamma$ to
$d\rho/\rho^2,\ d\vv/\rho$)
is
\begin{equation}\label{Kerrscsymbol}
(1+O(\rho)) \xi^2-4 \xi \ggamma (1-\vv/2) + 4((1-\vv/2)^2-1)\ggamma^2+O(\rho)
\end{equation}

Now we change coordinates to the ``correct'' system of $\rho,
v=\vv-\vv^2/4,$ in which the metric assumes the normal form.  We have
$$
\vv = 2(1-\sqrt{1-v}),
$$
hence in particular,
$$(1-\vv/2) = \sqrt{1-v},$$
while the vector fields are transformed by
$$
\pa_{\vv} \leadsto \sqrt{1-v} \pa_v,\quad
\pa_\rho \leadsto \pa_\rho,
$$
so that
$$
\ggamma \leadsto \sqrt{1-v}\gamma,\quad \xi \leadsto \xi.
$$
These changes yield the symbol
$$
\xi^2-4(1-v) \xi \gamma -4 v(1-v) \gamma^2 +O(\rho).
$$
Thus we find that in the notation of
\eqref{eq:inversemetriccomponents}, for the Kerr metric, we may read
off the coefficients of the dual metric in normal form as:
\begin{equation}\label{Kerrconstants}
\omega|_{\rho=v=0}=1,\ \alpha|_{\rho=v=0}=2,\ \beta|_{\rho=v=0}=4.
\end{equation}

\section{Explicit log terms}

In this section we describe how to explicitly compute the leading
order log singularity in the expansion at the radiation field face,
and verify that its coefficient is nonzero for the Kerr metric, whenever the
radiation field does not vanish identically.

Recall that we know a priori that if $\Box u=f$ with $f \in \dCI,$ then
$$
w\equiv \rho^{-(n-2)/2} u,
$$
which solves
$$
\rho^{-(n-2)/2-2} \Box \rho^{(n-2)/2} w=\rho^{-(n-2)/2-2} f\in \dCI
$$
has an expansion at the radiation field front face, i.e., locally in
the variables $s=\vlog/\rholog, \rholog, \ylog$ beginning
\begin{equation}\label{seriesansatz}
w\sim w_0(s,\ylog) + w_1^0(s,\ylog) \rholog+w_1^1(s,\ylog) \rholog\log
\rholog+w_1^2(s,\ylog) \rholog\log^2 \rholog.
\end{equation}
To explicitly find these terms (at least in principle) we recall that
we may write
$$
L=\rho^{-(n-2)/2-2} \Box \rho^{(n-2)/2}=L_0+\badmod^2,
$$
with
$$    L_0 = 4\pd[\vlog]\left( \rholog\pd[\rholog] + \vlog
      \pd[\vlog]\right). $$
We will need to analyze the module term more closely to obtain the
explicit singularity.  

To begin, we return to our original coordinate system $\rho, v, y$ and
note that if we look at the module vector fields $\rho \pa_v, v\pa_v,
\rho \pa_\rho,$ when we change to logified coordinates these become
respectively 
$$
\rholog\pa_\rholog+ \chi m \rholog (1+\log \rholog) \pa_{\vlog},\
(\vlog-\chi m\rholog \log \rholog)(1+\chi' m\rholog \log\rholog) \pa_{\vlog},\
\rholog(1+\chi' m \rholog \log\rholog) \pa_{\vlog}.
$$
We then perform the radiation field blowup $s=\vlog/\rholog$ and note
that the terms we get in this manner are spanned by $$\pa_s, \log\rholog\pa_s,
\rholog \pa_\rholog, \rholog \log \rholog \pa_{\rholog}.$$  Of these
terms, the important one for our purposes is $\log\rholog \pa_s,$ as
it is the only one that can produce a $\log\rholog$ term when
applied to a series of the form \eqref{seriesansatz}.  We now note the
crucial fact that in changing to log coordinates followed by lifting vector fields from $\cM,$ we have
\begin{align*}
\rho\pa_\rho &\leadsto \rholog \pa_{\rholog} + (\chi m-s) \pa_s + \chi
m
\log\rholog \pa_s,\\
v \pa_v &\leadsto (s-\chi m\log\rholog) (1+ \chi' m \rholog \log \rholog)
\pa_s,\\
\rho \pa_v &\leadsto (1+ \chi' m\rholog \log \rholog)\pa_s,
\end{align*}
hence isolating the crucial term, we simply remark that
\begin{equation}\label{vflifts}
\begin{aligned}
\rho\pa_\rho &\leadsto \chi m
\log\rholog \pa_s+\dots,\\
v \pa_v &\leadsto -\chi' m \log \rholog \pa_s+\dots,\\
\rho \pa_v &\leadsto \chi' m\rholog \log \rholog\pa_s+\dots.
\end{aligned}\end{equation}
(In dealing with $\CIlog$ coefficients of such terms, meanwhile, we
note that since every factor $\log\rholog$ also comes with a factor of
$\rholog,$ in analyzing the coefficient of $\log\rholog \pa_s$ in the
lift, it suffices to freeze these coefficients at $\rho=0.$)
We also recall that the operator
$$
4 \pa_v  (\rho \pa_\rho+ v \pa_v)-4 m \rho \pa_v^2
$$
lifts under this transformation to precisely
$$
L_0 = 4 \pa_s \pa_{\rholog}.
$$

Now we return to the form of a general long-range scattering metric.
Following the proof of Lemma~\ref{lemma:LmodM}, we can more precisely
write, using the notation of \eqref{eq:inversemetriccomponents} for
the dual metric components,
$$
L=4 \pa_v  (\rho \pa_\rho+ v \pa_v)-4 m \rho
\pa_v^2+\omega(\rho\pa_\rho)^2 + 2\alpha
\rho\pa_\rho v \pa_v+\beta (v\pa_v)^2+E
$$
where $E$ consists of first order terms in the module, second order
terms vanishing to higher order at $\rho=0,$ and terms
involving $\pa_y.$  Now lifting this expression to the logified,
blown-up space, using \eqref{vflifts}, it becomes
$$
L=4 \pa_s \pa_{\rholog} + m^2 (\omegabar-2\alphabar+\betabar)
\log^2 \rholog \pa_s^2 + E'
$$
where $E'$ consists of terms up to second order in
$\rholog\pa_\rholog,\ \rholog\log \rholog \pa_\rholog,\ \pa_s, \log
\rholog \pa_s$ with log-smooth coefficients, \emph{but containing at most one factor of this last
  vector field}, and $\omegabar,\alphabar,\betabar$ denote the
respective restrictions of these functions to $\rho=v=0.$   Because
the derivative of $\chi$ is supported away from the radial set, $E'$
also includes the error terms from dropping the factors of $\chi$ and
$\chi'$ in the above expression.

Now we apply this expression for $L$ to the series Ansatz
\eqref{seriesansatz}.  Matching the resulting coefficients of $\log^2
\rholog$ yields
$$
4 \pa_s w_1^2 + m^2 (\omegabar-2\alphabar+\betabar)\pa_s^2 w_0=0,
$$
hence, since all the coefficients vanish for $s \to -\infty,$ we may
integrate to find
$$
w_1^2=-\frac{m^2}4 (\omegabar-2\alphabar+\betabar)\pa_s w_0.
$$
For the particular case of the Kerr metric, \eqref{Kerrconstants} now gives
$$
w_1^2=-\frac{m^2}4\pa_s w_0.
$$
The function $w_0$ cannot be constant unless it is zero (again
since it vanishes for $s \ll 0$), so in general, we find that $w_1^2
\neq 0.$  (Note that $\pa_s w_0$ is in fact exactly the Friedlander
radiation field in this context.)

\def\cprime{$'$} \def\cprime{$'$}

\end{document}